\documentclass[a4paper,10pt]{article}

\usepackage{authblk}
\usepackage{amsmath}
\usepackage{amssymb}
\usepackage{amsthm}
\usepackage[mathscr]{eucal}
\usepackage{bm}
\usepackage{bbm}
\usepackage{enumerate}
\usepackage{a4wide}
\usepackage{tikz}
\usepackage{graphicx}
\usepackage{subcaption}
\usepackage{epstopdf}
\usetikzlibrary{decorations.pathreplacing}
\usepackage{algorithm}
\usepackage{algorithmic}
\usepackage{hyperref}
\usepackage[nocompress]{cite}

\newcommand{\pr}{\mathbb{P}}
\newcommand{\Prob}[1]{\pr\left(#1\right)}

\newcommand{\CProb}[2]{\pr\left(\left.#1\right|#2\right)}

\newcommand{\e}{\mathbb{E}}
\newcommand{\Exp}[1]{\e\left[#1\right]}
\newcommand{\Expn}[1]{\e_n\left[#1\right]}
\newcommand{\CExp}[2]{\e\left[\left.#1\right|#2\right]}

\newcommand{\abs}[1]{\left\vert#1\right\vert}

\newcommand{\To}{\rightarrow}

\newcommand{\dlim}{\ensuremath{\stackrel{d}{\rightarrow}}}

\newcommand{\ind}[1]{\mathbbm{1}_{\left\{#1\right\}}}

\newcommand{\ceil}[1]{\left\lceil{#1}\right\rceil}
\newcommand{\floor}[1]{\left\lfloor{#1}\right\rfloor}

\newcommand{\bigO}[1]{O\left(#1\right)}

\newcommand\numberthis{\addtocounter{equation}{1}\tag{\theequation}}
\allowdisplaybreaks

\newtheorem{theorem}{Theorem}[section]
\newtheorem{definition}{Definition}[section]
\newtheorem{lemma}[theorem]{Lemma}

\newtheorem{corollary}[theorem]{Corollary}
\newtheorem{assumption}[definition]{Assumption}

\newtheorem{remark}[theorem]{Remark}
\title{Average nearest neighbor degrees in scale-free networks}
\author[1]{Dong Yao}
\author[2]{Pim van der Hoorn}
\author[3,4]{Nelly Litvak}
\affil[1]{Duke University, North Carolina, United States}
\affil[2]{Northeastern University, Boston, United States}
\affil[3]{University of Twente, Enschede, the Netherlands}
\affil[4]{Eindhoven University of Technology, Eindhoven, the Netherlands}

\begin{document}

\maketitle

\begin{abstract}
The average nearest neighbor degree (ANND) of a node of degree $k$ is widely used to measure dependencies between degrees of neighbor nodes in a network. We formally analyze ANND in undirected random graphs when the graph size tends to infinity. The limiting behavior of ANND depends on the variance of the degree distribution. When the variance is finite,  the ANND has a deterministic limit. When the variance is infinite, the ANND scales with the size of the graph, and we prove a corresponding central limit theorem in the configuration model (CM, a network with random connections). As ANND proved uninformative in the infinite variance scenario, we propose an alternative measure, the average nearest neighbor rank (ANNR). We prove that ANNR converges to a deterministic function whenever the degree distribution has finite mean. We then consider the erased configuration model (ECM), where self-loops and multiple edges are removed, and investigate the well-known `structural negative correlations',  or `finite-size effects', that arise in simple graphs, such as ECM, because large nodes can only have a limited number of large neighbors. Interestingly, we prove that for any fixed $k$, ANNR in ECM converges to the same limit as in CM. However, numerical experiments show that finite-size effects occur when $k$ scales with $n$.
\end{abstract}

\section{Introduction}
The goal of this paper is to analytically derive the limiting properties of the average nearest neighbor degree (ANND) in a general class of random graphs. The motivation for this analysis is that the ANND is one of the commonly accepted measures for dependencies between degrees of neighbor nodes. Such dependencies are called degree-degree correlations or network assortativity. A network is said to be assortative, if the correlation between degrees of neighbor nodes is positive and disassortative when it is negative. In assortative networks, nodes of high degree have a preference to connect to nodes of high degree. When the network is disassortative, nodes of high degree have a connection preference for nodes of low degree~\cite{Newman2003}. If there is no connection preference, the network is said to have neutral mixing.

Currently, degree-degree correlations are part of the standard set of properties used to
characterize the structure of networks. See \cite{Noldus2015} for a survey of the work on network
assortativity. The effect of degree-degree correlations on disease spreading in networks has been extensively addressed in the literature, cf. \cite{Barthelemy2005,Boguna2003a,Boguna2003,Boguna2002}.
For instance, it was shown that disassortative networks are easier to immunize and a disease takes longer to spread in assortative networks \cite{DAgostino2012}. In the field of neuroscience, it was shown that assortative brain networks are better suited for signal processing~\cite{Schmeltzer2015}, while assortative neural networks are more robust to random noise \cite{Franciscis2011}.
Under attacks, when edges or vertices are removed, assortative networks appear to be more resilient than disassortive networks \cite{Newman2003,Vazquez2003}.
On the other hand, when different networks interact, assortativity actually decreases the robustness
of the whole system \cite{Zhou2012}.

A well established measure for degree-degree correlation computes Pearson's correlation coefficient on the joint data of the degrees at both sides of an edge \cite{Newman2002,Newman2003}. However, this measure has been shown to depend on the size of the network and it converges to zero when the degree distribution has infinite variance \cite{Litvak2013,Hofstad2014}. To remedy this, new measures have been introduced \cite{Hofstad2014,Hoorn2014a} that are based on rank-correlations such as Spearman's rho and Kendall's tau. These measures are shown to converge to a proper limit when the size of the network tends to infinity \cite{Hofstad2014,Hoorn2014}, see also \cite{Hoorn2016b} for an overview.

Both Pearson's correlation coefficient and rank-correlation measures represent the full degree-degree correlation structure of a network as \emph{one} number and hence, are unable to capture local correlation structures. Instead, another common approach in the literature is to compute for a node of degree $k$, the average degree of it's neighbors and then take the average of these values over all nodes of degree $k$. This gives us a function of $k$ called the Average Nearest Neighbor Degree (ANND) \cite{Pastor-Satorras2001, Catanzaro2005}. For assortative networks the ANND is an increasing function of $k$, while it is decreasing in $k$ for disassortative networks and constant when the network has neutral mixing.

Networks with neutral mixing are an essential tool in the analysis of degree-degree correlations. An important example of such networks are given by the configuration model \cite{Bollobas1980,chen2013,Molloy1995,Wormald1980}, which generates networks with a given degree sequence and lets the nodes connect randomly to each other. It is well known that the ANND in the configuration model is a constant, given by the  empirical second moment of the degrees, divided by the empirical first moment. However, until now,  no rigorous proof of this result was known, and the asymptotic behavior of ANND in the infinite graph size limit has not been analyzed. Given the wide application of ANND, this is an essential gap especially in the most realistic scenario when degrees have infinite variance. In this case, the empirical second moment grows with the size of the network and hence the ANND diverges as the network size tends to infinity. This means that, similarly to the Pearson's correlation coefficient, the ANND is not a consistent measure for degree-degree correlations in networks.

In this paper we address the convergence of the ANND in random graphs with given joint degree distribution of neighbor nodes, and in the configuration model.
In Section~\ref{sec:convergence_annd_general} we prove that, when the degree distribution has a finite variance, the ANND converges point-wise in probability in the general case. Moreover, for the configuration model the rate of convergence is uniform.
In Section~\ref{sec:convergence_annd_cm} we turn to the important case of infinite variance and focus on the configuration model. We prove a central limit theorem for the ANND where the limiting random
variable has a stable distribution with infinite variance. In Section~\ref{sec:annr} we apply the strategy of using ranks instead of degrees and propose the Average Nearest Neighbor Rank (ANNR) correlation measure. We prove that ANNR converges point-wise for any joint degree distribution. In the configuration model the limit is a constant, which is determined by the size-biased degree distribution. Moreover, the limit  is preserved, for any fixed $k$, in the erased configuration model, which is a simple graph obtained by removing self-loops and multiple edges.

Numerical experiments in Section~\ref{sec:experiments} illustrate our results for ANND and ANNR. One striking observation is that the empirical average of ANND and ANNR over several networks rapidly decrease with $k$, when $k$ is large. We explain this by the fact that for sufficiently large $k$, with positive probability, there is no node of such degree. When correcting by this probability, ANND and ANNR in the configuration model no longer decrease. Moreover, ANNR shows remarkably small fluctuations around its theoretical value even in small networks.

In the erased configuration model, for very large degrees $k$, which scale with the size of the graph, we observe a decline of ANNR. This phenomenon, called `structural correlations', is well known in the literature~\cite{Catanzaro2005,Barabasi2016}. It is explained by the fact that the graph is simple, therefore, nodes of large degrees are forced to connect to nodes of smaller degrees. Structural correlations for the rank-based correlation measure Spearman's rho have been observed before~\cite{Hoorn2015PhysRev}. Here we see that this phenomenon holds for the ANNR as well.

We start with introducing definitions and notations in Section~\ref{sec:notations_definitions}. Then in Section~\ref{sec:sampled}, before analyzing ANND, we establish an upper bound $K_n$ depending on $n$, such that all values $1,2,\ldots,K_n$ are present, with high probability, in an i.i.d. degree sequence of sequence length $n$, sampled from a regularly-varying distribution. To the best of our knowledge, this result is new and can be of independent interest.

\section{Notations and definitions}\label{sec:notations_definitions}

\subsection{Graphs and degree distributions}
Let $G_n=([n],E_n)$ denote an undirected graph of size $n$, with the nodes labeled from $1$ to $n$, and $E_n$ the set of edges. We write $D_i$ for the degree of node $i$ and call ${\bf D}_n = (D_1, \dots, D_n)$ the degree sequence of $G_n$. We use the term `degree sequence' to refer to any sequence $(D_1,\ldots,D_n)$ for which the sum $L_n=\sum_{i = 1}^n D_i$ is even.

For computations, it is convenient to replace each edge between nodes $i$ and $j$ by two directed edges $i \to j$ and $j \to i$. We denote by $G_{ij}$  the number of edges $i \to j$, and note that $G_{ij}$ can be larger than $1$ because we allow for multi-graphs. By construction we have $G_{ij} = G_{ji}$. Furthermore, $G_{ii}$ is \emph{twice} the number of self-loops of node $i$. With these notation the degree of a node is given as $D_i = \sum_{j = 1}^n G_{ij}$, which is the number of half edges attached to the node. We will use the notation $\sum_{i \to j}$ to denote the summation over all edges $i \to j$ in the graph (note that each pair of connected nodes in the undirected graph is counted twice in such summation).

In this paper, we use i.i.d. degree sequences, formally defined as follows~\cite{Hofstad2016}. Let $\mathscr{D}$ be a positive integer-valued random variable, with cumulative distribution function $F$ and probability density function $f$.
Let $D_1,D_2,\ldots,D_{n-1},d_n,$ be i.i.d. samples from $F$, and define:
\begin{equation*}
  D_n=d_n+\ind{(\sum_{i=1}^{n-1} D_i+d_n) \hspace{0.3em} is \hspace{0.3em} odd}\,.
\end{equation*}
We will write ${\bf D}_n = \texttt{IID}(\mathscr{D})$ for such sequence. Observe that the correction term in $D_n$ is uniformly bounded in $n$, therefore, it does not contribute to the asymptotic behavior of random graphs with degree sequence $\texttt{IID}(\mathscr{D})$. Hence, without loss of generality, we will consider the degrees $D_1,D_2,\ldots, D_n$ as i.i.d. samples from $\mathscr{D}$.

Throughout the paper we will denote the empirical degree density function by
\[f_n(k)=\frac{1}{n} \sum_{i=1}^n \ind{D_i = k},\quad k = 0, 1, \dots,\]
and the size-biased empirical degree density function will be denoted by
\[f_n^\ast(k)= \frac{1}{L_n} \sum_{i \to j} \ind{D_i = k} = \frac{1}{L_n}\sum_{i = 1}^n k \ind{D_i = k} = \frac{kf_n(k)}{L_n},\quad k = 1, 2, \dots.\]

In addition, we will denote by $F_n$ and $F_n^\ast$, the cumulative distribution functions corresponding to $f_n$ and $f_n^\ast$, respectively.

\subsection{Distances between probability measures}

In order to describe convergence of empirical distributions to their limits, we will use two different distances between probability measures: the total variation distance and the Kantorovich-Rubinstein distance.

Let $f$ and $g$ be two probability density functions on the non-negative integers and let $F$ and $G$ denote their respective cumulative distribution functions.

The total variation distance $d_{tv}$ is defined as
\begin{equation*}
d_{tv}(f,g) = \frac{1}{2}\sum_{k=0}^\infty \abs{f(k)-g(k)}\,.
\end{equation*}
Using the definition \[F(k)=\sum_{l \leq k} f(l),\]
we immediately see that
\begin{equation}\label{eq:sup_bound_total_variantion}
\sup_{k \ge 0}\abs{F(k)-G(k))} \leq 2 d_{tv}(f, g).
\end{equation}

The Kantorovich-Rubinstein distance $d_1$ between $f$ and $g$,  is defined as follows:
\begin{equation}\label{eq:def_kantorovich_distance}
d_1(f,g) = \sum_{k\ge 0} \abs{F(k)-G(k)}\,.
\end{equation}
Let $X$ and $Y$ be non-negative integer-valued random variables with distributions $F$ and $G$ respectively, and assume that $F$ and $G$ have a finite mean.  Then it follows from
\begin{equation}
\mathbb{E}(X)=\sum_{k \ge 0} \mathbb{P}(X > k)\,
\end{equation}
and the triangle inequality that
\begin{equation*}
\abs{\mathbb{E}(X)-\mathbb{E}(Y)} \leq d_1(F,G)\,.
\end{equation*}
In particular, convergence in $d_1$ implies convergence of the means.

We will usually use the Kantorovich-Rubinstein distance when the means are finite, and the total variation distance when the means are infinite.

\subsection{Regularly-varying distributions}
An important feature shared by many real-world networks is that their degree distribution is scale-free. This is often visualized by showing that $\Prob{D > k}$ behaves as an inverse power of $k$. Mathematically we can model this using regularly-varying distributions. In this paper, we assume that $\mathscr{D}$ has a regularly-varying density, i.e.
\begin{equation}\label{eq:def_regularly_varying_pmf}
	\Prob{\mathscr{D} = k} = l(k) k^{-\gamma-1}, \quad k = 1,2,\dots
\end{equation}
The parameter $\gamma>1$ is called the \emph{exponent} of the distribution and $l(x)$ is a slowly varying function, which means that for every $\lambda > 0$,
\[
	\lim_{x \to \infty} \frac{l(\lambda x)}{l(x)}=1.
\]
We will furthermore assume that the slowly-varying function $l(x)$ is eventually monotone.

Observe that if $\mathscr{D}$ satisfies \eqref{eq:def_regularly_varying} then $\Exp{\mathscr{D}^p} < \infty$ for all $p < \gamma$. We refer to \cite{Bingham1989} for a thorough treatment of regular variation. We do want to point out that due to Karamata's theorem, it follows from \eqref{eq:def_regularly_varying_pmf} that
\begin{equation}\label{eq:def_regularly_varying}
	\Prob{\mathscr{D} > t} = \tilde{l}(t) t^{-\gamma} \quad \text{for all } t > 0,
\end{equation}
where $\tilde{l}(x) \sim l(x)/\gamma$ as $x \to \infty$. See Lemma~\ref{lem:regularly_varying_cdf} from more details. Due to this asymptotic equivalence we shall slightly abuse notation and use $l(x)$ to denote both the slowly-varying function associated with $\Prob{\mathscr{D} = k}$ and $\Prob{\mathscr{D} > t}$. When $\mathscr{D}$ satisfies \eqref{eq:def_regularly_varying_pmf} we say that $\mathscr{D}$ is regularly varying with exponent $\gamma$.

\subsection{Stable distributions}

Stable distributions are important for us, since they come up as the limit distribution for central limit theorems involving regularly-varying distributions. A random variable $X$ is said to have a stable distribution if for every $n \ge 2$ there exists constants $a_n$ and $b_n$ such that for any sequence $X_1, \dots X_n$ of
independent copies of $X$,
\[
	X_1 + X_2 + \dots + X_n \stackrel{d}{=} a_n X + b_n.
\]
The characteristic function of stable distribution can be classified using four parameters $\alpha, \beta, \sigma$ and $\mu$, see \cite[Definition 1.1.6]{Schmeltzer2015}, and the corresponding random variable is denote by $S_\alpha(\sigma, \beta, \mu)$. The parameter $\alpha$ is called the stability index and is the most important parameter for our purposes, since it relates to the exponent $\gamma$ of the regularly-varying distribution. The Stable Law CLT (\cite[Theorem 4.5.1]{whitt2006}), states that for a sequence $(X_i)_{i \ge 1}$ of i.i.d. copies of a regularly-varying random variable with exponent $\gamma > 1$ there exist a slowly-varying function $l_0(n)$ such that
\begin{equation}\label{eq:def_stable_law}
	\frac{n^{-\frac{1}{\gamma}}}{l_0(n)} \sum_{i = 1}^n X_i \dlim S_{\gamma}(1, \beta, 0)
\end{equation}
When the $X_i$ are non-negative, as is the case for degrees, $\beta = 1$. Hence we will omit the dependence on the other three parameters and write $S_\gamma$ for a stable distribution with stability index $\gamma$. 

\subsection{Average nearest neighbor degree}
The Average Nearest Neighbor Degree (ANND) of nodes of degree $k$ in a graph $G_n$, of size $n$, is formally defined on simple graphs as (see also \cite{Catanzaro2005})
\begin{equation}\label{eq:def_Phi_n_alternative}
\Phi_n(k)=\ind{f_n(k) > 0}\sum_{\ell > 0} \ell P(\ell|k),
\end{equation}
where $P(\ell|k)$ is defined as the probability that a node of degree $k$ is connected to a node of degree $\ell$. The indicator $\ind{f_n(k) > 0}$ stands for the event that at least one node of degree $k$ exists in the network, otherwise, $\Phi_n(k)$ is set to zero. Let us define
\begin{equation}\label{eq:def_joint_degree_density}
	h_n(k,\ell)= \frac{1}{L_n} \sum_{i \to j} \ind{D_i = k, \, D_j = \ell},
\end{equation}
as the empirical joint distribution of the degrees on both sides of a randomly sampled edge in the graph $G_n$. We shall refer to $h_n(k,\ell)$ as the \emph{joint degree distribution}. 
Next, we note that $P(\ell|k)$ is equivalent to the probability that a randomly selected edge $i \to j$, conditioned on $D_i = k$, satisfies $D_j = \ell$, i.e. $P(\ell|k) = h_n(k,\ell)/f^\ast_n(k)$. Using this, we can extend \eqref{eq:def_Phi_n_alternative} to the setting of arbitrary (multi)graphs as
\begin{equation}\label{eq:def_Phi_n}
  \Phi_n(k) = \ind{f_n(k) > 0}\frac{\sum_{\ell > 0} h_n(k,\ell)\ell}{f_n^\ast(k)}.
\end{equation}

\subsection{Configuration model}
The configuration model (CM) \cite{Bollobas1980,chen2013,Molloy1995,Wormald1980} is an important model for generating graphs $G_n$ of size $n$, with a specific degree sequence. 
Since it generates graphs with neutral mixing (see e.g.~\cite{Hofstad2014,Hoorn2016b}), it is also a crucial model for the analysis of degree-degree correlations.

Given degree sequence ${\bf D}_n$, we assign $D_i$ stubs (half-edges) to each node $i$. Then we randomly pair the $L_n$ stubs to obtain a graph (possibly a multi-graph) $G_n$ with the given degree sequence. This procedure can be extended to generate graphs with a specific degree distribution. Let $\mathscr{D}$ be a non-negative, integer-valued random variable with probability density $f$. When ${\bf D}_n =\texttt{IID}(\mathscr{D})$, the configuration model generates random graphs, in which the empirical degree distribution $f_n$ converges to $f$ as $n\to\infty$.

CM can be adjusted to generate simple graphs. One approach is to simply repeat the wiring of the graph until the resulting graph is simple. This is the repeated configuration model (RCM). The RCM can be applied successfully only if the probability to obtain a simple graph converges to a nonzero value as $n\to\infty$. It is well-known (see \cite[Chapter 7]{Hofstad2016}) that this is indeed the case if and only if $\mathscr{D}$ has finite second moment. Another way to obtain a simple graph is to simply remove all self-loops and replace all multiple edges between $i$ and $j$ by a single edge. This model, called the erased configuration model (ECM), generates a simple graph with correct asymptotic degree distribution whenever ${\bf D}_n=\texttt{IID}(\mathscr{D})$ and $\mathscr{D}$ has a finite mean~(see \cite[Theorem~7.10]{Hofstad2016}).

\section{Sampled degrees}
\label{sec:sampled}
Recall that in \eqref{eq:def_Phi_n} we set $\Phi_n(k) = 0$ if the degree
sequence contains no nodes of degree $k$. Therefore, before we start our analysis of the behavior of
$\Phi_n(k)$, we need to understand which $k$'s are present in the degree sequence ${\bf D}_n =\texttt{IID}(\mathscr{D})$. The
following result for regularly-varying distributions is, to the best of our knowledge, not known in the literature and can be of independent interest.

\begin{theorem}\label{thm:degree_sequence}
Let ${\bf X}_n = \{X_1, \dots, X_n\}$, be independent copies of an integer-valued random variable $X$, with regularly-varying probability mass function $f$ with exponent $\gamma > 1$ and suppose that $f(k)>0$ for all $k>0$. Then for any $0 < a < \frac{1}{\gamma+1}$
\begin{equation*}
\lim_{n\To \infty} \Prob{\left\{1,2,\ldots,\ceil{n^a}\right\} \subseteq {\bf X}_n} = 1\,.
\end{equation*}
On the other hand, if $a>\frac{1}{\gamma+1}$, then
\begin{equation*}
	\lim_{n\To \infty} \Prob{\ceil{n^{a}} \in {\bf X}_n}=0\,.
\end{equation*}
\end{theorem}

\begin{proof} 
Throughout the proof $l(k)$ denotes the slowly-varying function associated with the probability mass function of $X$.  
We start with the first statement.

Since $a < 1/(\gamma+1)$, it follows from Potter's bounds for slowly varying functions, that there exist $\gamma+1<b<\frac{1}{a}$, $C>0$ and $K \ge 1$, such that $f(k) \ge C k^{-b}$ for all $k \ge K$. Then, for sufficiently large $n$,
\begin{align*}
  &\hspace{-30pt}\Prob{\{ 1,2,\ldots, \ceil{n^a}\}\subseteq \{X_1,X_2,\ldots,X_n\}} \\
  &\geq 1-\sum_{k=1}^{\ceil{n^a}}\Prob{k\notin\{X_1,X_2,\ldots,X_n\}}  \\
  &\ge 1 - \sum_{k = 1}^{K-1} (1 - \Prob{X = k})^n - \sum_{k = K}^{\ceil{n^a}}(1-C k^{-b})^n\\
  &\geq 1 - (K-1)(1 - \Prob{X = K - 1})^n - (n^a + 1) (1-C n^{-ab})^n.
\end{align*}
Because $\Prob{X = K - 1} < 1$ we have $(K-1)(1 - \Prob{X = K - 1})^n \to 0$ as $n \to \infty$, while $ab<1$ implies that $(n^a + 1)(1-C n^{-ab})^n\to 0$ and hence $\Prob{\{ 1,2,\ldots, \ceil{n^a}\}\subseteq \{X_1,X_2,\ldots,X_n\}} \to 1$.

We follow a similar approach for the second statement. Since $\frac{1}{a} > \gamma + 1$, it follows from Potter's bounds that there exist $\frac{1}{a}<b^\prime<\gamma+1$, $C^\prime > 0$ and $K \ge 1$, such that $f(k) \leq$ $C^\prime k^{-b^\prime}$ for all $k \ge K$. Then, for sufficiently large $n$ we obtain
\begin{align*}
 	\Prob{\ceil{n^a} \in \{X_1,\ldots,X_n\}}&=1-(1-f(\ceil{n^a}))^n\\
 	&\leq 1-(1-Cn^{-ab^\prime})^n.
\end{align*}
Since $ab^\prime>1$, we have $\lim_{n\To \infty}(1-Cn^{-ab^\prime})^n=1$, which gives the result.
\end{proof}

\begin{remark}
Theorem \ref{thm:degree_sequence} can be adjusted in a straight-forward way to the case where the probability density function $f$ of $X$ has infinite support but is zero on some finite subset of the positive integers.
\end{remark}

It is interesting to notice that the degrees higher than $n^{\frac{1}{\gamma+1}}$ will appear in ${\bf D}_n$, in particular the maximal degree scales as $n^{\frac{1}{\gamma}}$. However, only up to $n^a$ with $a<\frac{1}{\gamma+1}$ one can guarantee that {\it all } degrees between $1$ and $n^a$ will participate in ${\bf D}_n$. In other words, for $\frac{1}{\gamma+1}<a<b<\frac{1}{\gamma}$ some values $k\in [n^a, n^b]$ will be missing in ${\bf D}_n$.

\section{Limiting behavior of $\Phi_n$ in graphs with general joint degree distributions and finite second degree}\label{sec:convergence_annd_general}

In this section we analyze the graphs with given limiting joint degree distribution $h(k,\ell)$ and finite second moment. Our main result (Theorem \ref{thm:convergence_annd}) proves the convergence of
the average nearest neighbor degree $\Phi_n$ to its limit
\begin{equation}\label{eq:annd_limit}
\Phi(k) = \ind{f^\ast(k) > 0}\frac{\sum_{\ell = 1}^{\infty} h(k,\ell)\ell}{f^\ast(k)},
\end{equation}
where
\[
	f^\ast(k) = \frac{k f(k)}{\Exp{\mathscr{D}}}.
\]
is the size-biased probability density function $f^\ast$ of a non-negative integer random variable $\mathscr{D}$. Note that for $k \ge 1$, $f(k) > 0$ if and only if $f^\ast(k) > 0$.

As commonly accepted in the random graph literature, we impose regularity assumptions on the degree sequences.  Assumption~\ref{asmp:regularity_degrees} uses the distance $d_1$ to simultaneously state the convergence of $f_n$, $f_n^*$ and their expectations to the corresponding limiting values. Note that convergence of the expectation of $f_n^*$ is equivalent to the convergence of the second moment of $f_n$. We denote by $A^c$ the event complementary to the event $A$.

\begin{assumption}[Regularity of the degrees]\label{asmp:regularity_degrees}
There exist a probability density $f$ on the non-negative integers with size biased version $f^\ast$ and
some $\alpha,\varepsilon > 0$, such that if we set
\[
	\Omega_n=\{\max\left\{d_1(f_n,f), \, d_1(f_n^*,f^*)\right\} \le n^{-\varepsilon}\},
\]
then, as $n \to \infty$,
\begin{equation*}
\Prob{\Omega_n^c} = O\left(n^{-\alpha}\right).
\end{equation*}
\end{assumption}

Assumption~\ref{asmp:regularity_degrees} looks slightly stronger than just convergence because it requires that the distances between the empirical and limiting distributions are not larger than a negative power of $n$ with high probability. However, this is not restrictive. In particular, Theorem~\ref{thm:convergence_degree_densities} below states that Assumption \ref{asmp:regularity_degrees} holds for ${\bf D}_n=\texttt{IID}(\mathscr{D})$ whenever $\mathscr{D}$ has a finite $(2+\eta)$-moment for some $\eta>0$. The proof, which
can be found in Section \ref{ssec:proof_convergence_degree_densities}, is a straightforward extension of the proof of \cite[Theorem 3.1]{Hoorn2016b}.

\begin{theorem}\label{thm:convergence_degree_densities}
Suppose that $\Exp{\mathscr{D}^{2+\eta}}<\infty$ for some $0 < \eta < 1$. Let ${\bf D}_n =
\emph{\texttt{IID}}(\mathscr{D})$, $0 < \varepsilon \le \frac{\eta}{2(\eta+2)}$ and $\Omega_n$ be defined as in Assumption
\ref{asmp:regularity_degrees}. Then, as $n \to \infty$,
\begin{equation*}
	\Prob{\Omega_n^c} = O\left(n^{-\varepsilon}\right)
\end{equation*}
so that, in particular, ${\bf D}_n$ satisfies Assumption \ref{asmp:regularity_degrees} with $\alpha = \varepsilon$.
\end{theorem}

We remark that the above result still holds if, instead of $\Omega_n$, we consider the event
\[
	\{ \max\{ d_1(f_n,f),d_1(f_n^*,f^*) \}\leq K n^{-\varepsilon} \},
\]
for any $K > 0$.

The second regularity assumption imposes, in a similar fashion, the convergence of the joint degree distribution $h_n(k,\ell)$, of the
degrees at both ends of a randomly selected edge.

\begin{assumption}[Regularity of the joint distribution]\label{asmp:regularity_structure}
There exists a joint probability density function $h(k,\ell)$ on the positive integers and some
$\kappa>0$, such that if
\begin{equation*}
	\Gamma_n= \left\{\sum_{k, \ell =1}^{\infty} \abs{h_n(k,\ell)-h(k,\ell)} \ind{f_n(k)>0}
	\leq n^{-\kappa}\right\},
\end{equation*}
then we have $\lim_{n\To \infty} \Prob{\Gamma_n} = 1$.
\end{assumption}

Assumption \ref{asmp:regularity_structure} is satisfied for ${\bf D}_n=\texttt{IID}(\mathscr{D})$ in several models, for example, the Configuration Model (see
\cite[Proposition 6.2]{Hoorn2016b}) and the Maximally Disassortative Graph Algorithm (see \cite[Theorem 3.3]{Hoorn2016}).

We are now ready to state the main result of this section.

\begin{theorem}\label{thm:convergence_annd}
Let $\{G_n\}_{n \ge 1}$ be a sequence of graphs, which satisfies Assumptions \ref{asmp:regularity_degrees} and
\ref{asmp:regularity_structure} and assume that the limiting distribution $f$ has a finite $(2+\eta)$-moment for some $0 < \eta < 1$. Then, for each fixed $k$ such that $f(k)>0$ and each $0 < \delta < min\left\{ \varepsilon, \frac{\kappa \eta}{\eta+1}\right\} $,
\begin{equation*}
 \lim_{n\To\infty} \Prob{\abs{\Phi_n(k) -{\Phi}(k) }\leq n^{-\delta}} =1\,.
\end{equation*}
\end{theorem}

The proof is based on splitting of $\abs{\Phi_n(k) -{\Phi}(k) }$ in several terms and bounding each of them  separately, using Assumptions~\ref{asmp:regularity_degrees} and
\ref{asmp:regularity_structure}. We give the proof in Section~\ref{ssec:proof_annd_general_random_graphs}.

\section{Limiting behavior of $\Phi_n$ in Configuration Models}\label{sec:convergence_annd_cm}

In \cite{Hofstad2014,Hoorn2014} it was shown that different measures
for degree-degree correlations converge to zero, as $n \to \infty$, for both the multi-graph configuration model, as
well as the repeated and erased versions. Therefore, one would expect that $\Phi_n(k)$ converges to some constant, independent of $k$. We will show that this is
indeed the case for all three models when we consider ${\bf D}_n=\texttt{IID}(\mathscr{D})$, where
$\mathscr{D}$ has finite $(2 + \eta)$ moment. When the second moment is infinite and $\mathscr{D}$ is regularly
varying, we establish a central limit theorem, where the limiting random variable has a stable distribution.

\subsection{Multi-graphs, finite variance of the degrees}

We first consider the case when $\mathscr{D}$ has finite $(2+\eta)$-moment. In \cite{Hoorn2016b} it is proven
that the directed version of the configuration model satisfies Assumption
\ref{asmp:regularity_structure}, where $h(k, \ell) = f^\ast(k)f^\ast(\ell)$. The proof can be
adjusted in a straight-forward manner to the undirected case. If we plug this result for $h(k, \ell)$
into \eqref{eq:annd_limit} and define $\nu_p = \Exp{\mathscr{D}^p}$ for $p = 1, 2$, we get
\[
	\Phi(k) = \frac{\sum_{\ell > 0} f^\ast(k)f^\ast(\ell)\ell}{f^\ast(k)}
	= \sum_{\ell > 0} \frac{f(\ell) \ell^2}{\nu_1} = \frac{\nu_2}{\nu_1}.
\]
Hence, it follows from Theorem \ref{thm:convergence_annd} that
\begin{equation}\label{eq:convergence_annd_cm_weak}
	\lim_{n \to \infty} \Prob{\abs{\Phi_n(k) -\frac{\nu_2}{\nu_1}} > n^{-\delta}} = 0.
\end{equation}

With a little extra work, we can prove the following stronger result, which states that
the convergence in probability is uniform in $k$, and gives an upper bound on the speed of
convergence. The proof exploits the construction of CM and is provided in Section~\ref{sec:proofs_cm}.
\begin{theorem}\label{thm:convergence_annd_cm_strong}
Let $\mathscr{D}$ be an integer-valued random variable which satisfies $\Exp{\mathscr{D}^{2+\eta}} <
\infty$, for some $0 < \eta < 1$, and let $\{G_n\}_{n \ge 1}$ be a sequence of graphs generated by
CM with ${\bf D}_n=\emph{\texttt{IID}}(\mathscr{D})$. Then, for any $0<\delta < \frac{\eta}{2(\eta+2)}$,
\[
	\sup_{k \ge 0} \, \Prob{ \abs{\Phi_n(k) -\frac{\nu_2}{\nu_1}}1_{\{ f_n(k)>0 \}} > n^{-\delta}}
	 = O\left(n^{-\frac{\delta}{2}} + n^{\delta- \frac{\eta}{2(\eta + 2)}}\right).
\]
\end{theorem}

\subsection{Multi-graph, infinite variance of the degrees}

It is important to note that the finite second moment condition in Theorem
\ref{thm:convergence_annd_cm_strong} can not be relaxed to the case where $\mathscr{D}$ has only
finite mean, since then $\sum_{\ell > 0} \ell^2 f(\ell)$ is no longer finite and hence
$\sum_{\ell > 0} \ell^2 f_n(\ell)$ diverges as $n \to \infty$. Consequently $\Phi_n(k)$ will
increase as $n \to \infty$.

In order to understand how $\Phi_n(k)$ scales with $n$ observe that for CM we have
\[
	\Phi_n(k) = \frac{\sum_{\ell > 0} h_n(k,\ell)\ell}{f_n^\ast(k)}
	\approx \frac{1}{L_n}\sum_{\ell > 0} \ell^2 f_n(\ell) = \sum_{i = 1}^n \frac{D_i^2}{L_n}.
\]
When $D_i$ are sampled from a regularly-varying random variable with exponent $1 < \gamma < 2$, it
follows that $\sum_{i = 1}^n D_i^2$ scales as $n^{2/\gamma}$ and $L_n = \sum_{i = 1}^n D_i \approx n\Exp{\mathscr{D}}$.
It now follows that $\Phi_n(k)$ scales as $n^{2/\gamma - 1}$. These scaling terms can be made exact by
adding slowly-varying functions, and we show in Theorem \ref{thm:annd_clt_configuration_model}
that $\Phi_n(k)$ rescaled with $n^{2/\gamma - 1}$ converges to a random variable with a stable
distribution.

Before formulating the result, we need a weaker version of Assumption \ref{asmp:regularity_degrees}, for this does not hold anymore when the second moment of the degrees is infinite.
\begin{assumption}\label{asmp:convergence_distribution_first_moment}
There exist a probability density $f$ on non-negative integers with size biased version $f^\ast$ and
some $\alpha,\varepsilon > 0$, such that if we set
\[
	\Omega_n:=\{ \max\{ d_1(f_n,f),d_{tv}(f_n^*,f^*) \}\leq n^{-\varepsilon} \},
\] 
then, as $n \to \infty$,
\begin{equation*}
  \Prob{\Omega_n^c}=O(n^{-\alpha})\,.
\end{equation*}
\end{assumption}
Observe that this assumption resembles Assumption \ref{asmp:regularity_degrees} with the only exception that we replaced the Kantorovich-Rubenstein distance $d_1$ for the size-biased degree distribution $f^\ast$ by the total variation distance $d_{tv}$. This is because the convergence of the Kantorovich-Rubinstein distance implies the convergence of the first moment, which does not hold for the size-biased distribution when $\mathscr{D}$ has infinite variance. The next theorem from \cite{Hoorn2016b} shows that Assumption~\ref{asmp:convergence_distribution_first_moment} holds in CM with ${\bf D}_n=\texttt{IID}(\mathscr{D})$.

\begin{theorem}[\cite{Hoorn2016b} Theorem 3.1]\label{thm:convergence_distribution_first_moment}
Suppose that  $\mathbb{E}(\mathscr{D}^{1+\eta})< \infty$, for some $0 < \eta < 1$. Let ${\bf D}_n={\emph{\texttt{IID}}}(\mathscr{D})$,
$0<\varepsilon \leq \frac{\eta}{4(\eta+2)}$ and $\Omega_n$ as defined in Assumption \ref{asmp:convergence_distribution_first_moment}. Then, as $n \to \infty$,
\begin{equation*}
  \mathbb{P}(\Omega_n^c)=O(n^{-\varepsilon})\,,
\end{equation*}
so that, in particular, ${\bf D}_n$ satisfies Assumption \ref{asmp:convergence_distribution_first_moment} with $\alpha = \varepsilon$.
\end{theorem}
Similar to Theorem \ref{thm:convergence_degree_densities}, this result still holds if we consider the upper bound $K n^{-\varepsilon}$ in $\Omega_n$, for some $K > 0$.

With this result we can now state the following central limit theorem for $\Phi_n$ in the configuration model with regularly-varying degrees. 

\begin{theorem}\label{thm:annd_clt_configuration_model}
Let $\mathscr{D}$ be an integer-valued regular varying-random variable with exponent $1 < \gamma < 2$ and let
$\{G_n\}_{n \ge 1}$ be a sequence of graphs generated by CM with ${\bf D}_n = \emph{\texttt{IID}}(\mathscr{D})$. 
Assume $f(k)>0$ for all $k>0$. Then there exists a slowly varying function $l_0(n)$, such that
$\frac{\Phi_n(k)}{l_0(n) n^{\frac{2}{\gamma}-1}}$ converges in distribution to $S_{\gamma/2}$, having stable distribution with stability index $\frac{\gamma}{2}$. More precisely, for any $\tau < \frac{1}{\gamma+1}$ and any bounded Lipschitz function $g$,
\begin{equation*}
  \lim_{n\To \infty} \sup_{1\leq k \leq n^{\tau}} \abs{\Exp{g\left(\frac{\Phi_n(k)}{l_0(n) n^{\frac{2}{\gamma}-1}}\right) -g\left(S_{\gamma/2}\right)}}=0.
\end{equation*}
\end{theorem}

\begin{remark}
Although the ANND does not converge to a constant for CM with infinite variance degrees, it is worth noting that the distribution of the limit random variable $S_{\gamma/2}$ is independent of $k$. 
This reflects the independence between the degree of connected nodes in the CM.
\end{remark}

The proof of Theorem \ref{thm:annd_clt_configuration_model} is given in Section~\ref{sec:proofs_cm}. This result has two important consequences. First, it shows that, up to some slowly-varying function, $\Phi_n(k)$ scales as $n^{\frac{2}{\gamma}-1}$ and hence $\Phi_n(k) \to \infty$ with $n$. Second, since the rescaled limit is a random variable, even for large $n$ the value of the rescaled $\Phi_n(k)$ can be significantly different for different graphs generated by the configuration model. Therefore, we conclude that the ANND is not a proper measure for graphs where the degree distribution has infinite second moment.

\subsection{Repeated configuration model}

Recall that RCM repeatedly generates CM graphs until the resulting graph is simple. Let $S_n$ denote the event that the graph $G_n$, generated by CM with ${\bf D}_n=\texttt{IID}(\mathscr{D})$, is simple. Then $\Prob{S_n}$ converges to a positive number  whenever $\mathscr{D}$ has finite second moment (Theorem 7.12 in \cite{Hofstad2014}). Therefore, in the rest of this section we assume that $\mathbb{E}({\mathscr{D}}^{2+\eta})  <\infty$ for some $\eta>0$.
Since for all $k$ with $f(k)>0$ it holds that
\begin{equation*}
\CProb{\abs{\Phi_n(k)-\frac{\nu_2}{\nu_1}}>n^{-\delta}}{S_n} \leq \frac{ \mathbb{P}\left(\abs{\Phi_n(k)-\frac{\nu_2}{\nu_1}}>n^{-\delta}\right)}{\mathbb{P}(S_n)}\,,
\end{equation*}
the result of Theorem \ref{thm:convergence_annd_cm_strong} can be extended in a trivial way to the RCM. We state it here for completeness, where we exclude the speed of convergence since it depends on the convergence of $\Prob{S_n}$.

\begin{theorem}\label{thm:convergence_annd_rcm_strong}
Let $\mathscr{D}$ be an integer-valued random variable which satisfies $\Exp{\mathscr{D}^{2+\eta}} <
\infty$, for some $0 < \eta < 1$, and let $\{G_n\}_{n \ge 1}$ be a sequence of graphs generated by
the RCM with ${\bf D}_n=\emph{\texttt{IID}}(\mathscr{D})$. Then, for any $0 < \delta < \frac{\eta}{2(\eta+2)}$,
\[
	\lim_{n \to \infty} \sup_{k} \Prob{ \abs{\Phi_n(k) -\frac{\nu_2}{\nu_1}}1_{\{ f_n(k)>0 \}} > n^{-\delta}}
	= 0.
\]
\end{theorem}

\subsection{Erased configuration model}

Unlike the RCM, the ECM is well-defined and has the correct limiting degree distribution when degrees have only finite
expectation. Therefore, the ECM is used for generating simple graphs with infinite variance of the degrees. Moreover, due to computational advantages of the ECM over the RCM, the ECM is also preferred when the degree distribution has finite variance. In this section we consider both finite and infinite variance scenario.

Because of the removal of edges, the eventual degree sequence in the ECM is, in general, different from ${\bf D}_n = \texttt{IID}(\mathscr{D})$. We will use symbols with hats to denote characteristics of the ECM. For instance, $\widehat{D}_i$ denotes the degree of node $i$ in the ECM, while $D_i$ denotes its degree before the removal of edges.  Similarly, $\widehat{\Phi}_n(k)$ is the ANND in the ECM, while ${\Phi}_n(k)$ is computed on the multi-graph CM.

Our main result in this section establishes the scaling of the error between $\Phi_n$ and $\widehat{\Phi}_n$.
The proof is provided in Section~\ref{sec:proof_ecm}.

\begin{theorem}\label{thm:annd_erased_model_error_term}
Let $\mathscr{D}$ be regularly varying with exponent $\gamma >1$
and let $\{G_n\}_{n \ge 1}$ be a sequence of graphs generated by
the ECM with ${\bf D}_n=\emph{\texttt{IID}}(\mathscr{D})$. Set $a = (\gamma - 1)^2/(2\gamma) > 0$ and let $k$ be such that
$f^\ast(k) > 0$. If $\gamma\in (1,2)$, then
\[
	\lim_{n \to \infty} \Prob{\left|\widehat{\Phi}_n(k) - \Phi_n(k)\right|
	> n^{\frac{2}{\gamma} - 1 - a}} = 0,
\]
while for $\gamma > 2$ it holds that
\[
	\lim_{n \to \infty} \Prob{\left|\widehat{\Phi}_n(k) - \Phi_n(k)\right|
	> n^{-a}} = 0.
\]
\end{theorem}

The first result states that the difference $\left|\widehat{\Phi}_n(k)-\Phi_n(k)\right|$, although not necessarily vanishing, is of the smaller order than
$\Phi_n(k)$ in Theorem \ref{thm:annd_clt_configuration_model}. Therefore, we can extend the latter theorem to the ECM as follows.

\begin{theorem}
Let $\mathscr{D}$ be regularly varying with exponent $1 < \gamma < 2$ and let $\{G_n\}_{n \ge 1}$ be a sequence of graphs generated by
the ECM with ${\bf D}_n=\emph{\texttt{IID}}(\mathscr{D})$. Then there exists a slowly-varying function $l_0(n)$ and random variable $S_{\gamma/2}$ 
having a stable distribution with shape parameter $\frac{\gamma}{2}$, such that for each $k$ with $f^\ast(k)>0$,
\begin{equation*}
  \frac{\widehat{\Phi}_n(k)}{l_0(n)n^{\frac{2}{\gamma}-1}} \dlim S_{\gamma/2} \quad \text{as } n \to \infty.
\end{equation*}
\end{theorem}

The second part of Theorem \ref{thm:annd_erased_model_error_term} states that the difference $|\widehat{\Phi}_n(k) - \Phi_n(k)|$ goes to zero as $n \to \infty$, thus, using \eqref{eq:convergence_annd_cm_weak}, we obtain the following result.

\begin{theorem}\label{thm:convergence_annd_ecm_weak}
Let $\mathscr{D}$ be regularly varying with exponent $\gamma > 2$ and let $\{G_n\}_{n \ge 1}$ be a sequence of graphs generated by
ECM with ${\bf D}_n=\emph{\texttt{IID}}(\mathscr{D})$. Then for any $k$ such that $f^\ast(k) > 0$ and any
$0 < \delta < \frac{\gamma - 1}{2(\gamma + 3)}$, we have
\[
\lim_{n \to \infty} \Prob{\abs{\widehat{\Phi}_n(k) -\frac{\nu_2}{\nu_1}} > n^{-\delta}} = 0.
\]
\end{theorem}

The proof of Theorem~\ref{thm:convergence_annd_ecm_weak} follows directly by splitting
\[
	\abs{\widehat{\Phi}_n(k) -\frac{\nu_2}{\nu_1}} \le \abs{\Phi_n(k) -\frac{\nu_2}{\nu_1}}
	+ \left|\widehat{\Phi}_n(k) - \Phi_n(k)\right|
\]
and then using Theorem \ref{thm:convergence_distribution_first_moment} and Theorem \ref{thm:annd_erased_model_error_term}. Observe that when $\gamma > 2$, we have $\eta := (\gamma - 2)/2 > 0$ and  $\Exp{\mathscr{D}^{2 + \eta}} < \infty$, which yields the bound for $\delta$ in the above theorem.

Interestingly, in the literature, it has been observed that the ANND in the ECM decreases for large $k$. This phenomenon, termed as `structural correlations' or `finite-size effects' has been ascribed to the fact that the graph is simple, thus larger nodes do not have sufficient number of large neighbors~\cite{Barabasi2016,Catanzaro2005}. Theorem~\ref{thm:annd_erased_model_error_term} says that asymptotically the difference between ANND in CM and ECM is vanishing for any {\it fixed} $k$. However, when $k$ scales as a positive power of $n$, the difference between ANND in CM and ECM might be indeed non-vanishing because of the more significant contribution of self-loops and multiple edges. In numerical experiments we observed that there is indeed a difference between the range of ANND in CM and ECM. However, the numerical results showed enormous fluctuations even in the case of finite variance of the degrees, see e.g. Figure~\ref{fig:annd_cm_corrected}. Therefore, the numerical comparison of CM and ECM was inconclusive and is omitted in this paper. Exact evaluation of the finite-size effects remains an interesting open question.

\section{Average nearest neighbor rank}\label{sec:annr}
Many real networks exhibit power law degree distributions with infinite variance. As we saw above, the ANND functional in  CM with ${\bf D}_n=\texttt{IID}(\mathscr{D})$ degree sequences has the obvious disadvantage that it scales as the network size becomes larger, which makes it difficult to compare networks that have similar structure but different sizes. Also, in the infinite variance case, the scaled limit is a proper  random variable, which can vary as the degree sample changes.

Therefore, in this section we introduce a new measure that is suitable in the infinite variance case. A classical approach is to use rank-based measures. Thus we introduce the average nearest neighbor rank (ANNR), denoted by $\Theta_n(k)$, which is defined as follows:
\begin{equation*}
  \Theta_n(k)=\ind{f_n(k)>0}\frac{\sum_{\ell > 0} h_n(k,\ell)F_n^\ast(\ell)}{f_n^\ast(k)}\,.
\end{equation*}

The difference with ${\Phi}_n(k)$ is that we replace $\ell$ in the summation on the right-hand side of \eqref{eq:def_Phi_n} by $F_n^\ast(\ell)$. Recall that $F_n^\ast$ is the cumulative size-biased degree distribution, i.e.,
\[
	F_n^\ast(\ell) = \frac{1}{L_n} \sum_{i = 1}^n D_i \ind{D_i \le \ell}.
\] 
Hence, the value $F_n^\ast(\ell)$ can be seen as assigning a rank to edges involving a node of degree $\ell$. For instance, for the largest degree $D_{\max}$, $F_n^\ast(D_{\max}) = 1$, while for the smallest degree $D_{\min}$, $F_n^\ast(D_{\min})$ is the fraction of edges in the network with one of the nodes having degree $D_{\min}$. Recall that the total fraction of edges with one node having degree $k$ equals $f_n^\ast(k)$. Therefore, recalling that $P(\ell|k) = h_n(k,\ell)/f_n^\ast(k)$ is the probability that a node of degree $k$ is connected to a node of degree $\ell$, the function $\Theta_n(k)$ computes that average value of $F_n^\ast$ (rank) over all neighbors of nodes with degree $k$.

Note that $F_n^\ast(\ell)$ is a {\it non-decreasing} function of the degree, that is the degrees are ranked from low to high. Then, just as ${\Phi}_n(k)$, the ANNR $\Theta_n(k)$ tends to be increasing for assortative networks, and decreasing for disassortative networks. When the wiring is neutral, such as in the configuration model, the ANNR should be a constant, the value of which we will derive below.

Note that due to the straightforward equality $\sum_{\ell > 0} h_n(k,\ell)=f_n^*(k)$, the value of $\Theta_n(k)$ always lies in the interval $[0,1]$. This solves the scaling problem.
Further, the results in this section require only that $\mathscr{D}$ has finite ($1+\eta$)-moment.

Similarly to the ANND, we define the limiting version of the ANNR:
\begin{equation}
	\Theta(k) = \ind{f(k) > 0}\frac{\sum_{\ell > 0} h(k,\ell)F^\ast(\ell)}{f^\ast(k)}.
\end{equation}
The next theorem establishes convergence of $\Theta_n(k)$ to $\Theta(k)$. See Section \ref{ssec:proofs_annr} for the proof.

%From Theorem 6.1, using the inequality (6) we immediately get the next corollary.
%\begin{corollary}
%Under Assumption 6.1,we have
%\begin{equation*}
%  \mathbb{P}(\sup_k \abs{F_n^*(k)-F^*(k)}\leq n^{-\ep} ) \geq 1-O(n^{-\ep})\,.
%\end{equation*}
%\end{corollary}

\begin{theorem}[General convergence of ANNR]\label{thm:annr_convergence_general}
Let $\{G_n\}_{n \ge 1}$ be any sequence of graphs of size $n$ that satisfies Assumptions \ref{asmp:regularity_structure} and \ref{asmp:convergence_distribution_first_moment}.
Then, for each fixed $k$ such that $f(k)>0$ and each $0 < \delta <
\min\{\frac{\eta}{8+4\eta},\kappa \}$,
\begin{equation*}
\lim_{n\To\infty}\Prob{\abs{\Theta_n(k)-\Theta(k)} > n^{-\delta}}=0.
\end{equation*}
\end{theorem}

\begin{remark}
Since $|\Theta_n(k)|, \, |\Theta(k)| \le 1$, by dominated convergence we have that
\[
	\lim_{n \to \infty} \Exp{\abs{\Theta_n(k)-\Theta(k)}} = 0.
\]
\end{remark}

Comparing Theorem \ref{thm:annr_convergence_general} to Theorem \ref{thm:convergence_annd} we see that convergence of the ANNR holds for any random variable with finite $(1 + \eta)$-moment, instead of finite $(2 + \eta)$-moment of $\mathscr{D}$. This makes the ANNR a more suitable measure for degree-degree correlations than the ANND.

Now consider CM with ${\bf D}_n=\texttt{IID}(\mathscr{D})$. Then we have that
$h(k,\ell) = f^\ast(k) f^\ast(\ell)$ and hence we can write
\[
	\Theta(k) = \sum_{\ell > 0} f^\ast(\ell) F^\ast(\ell) = \Exp{F^\ast(\mathscr{D}^\ast)},
\]
where $\mathcal{D}^\ast$ has the cumulative distribution function $F^\ast$. Similarly to the ANND, we can directly apply Theorem \ref{thm:annr_convergence_general} to get that $\Theta_n(k)$ in the CM converges in probability to $\Exp{F^\ast(\mathscr{D}^\ast)}$, for all $k$ such that $f(k) > 0$. In addition, we can prove a stronger result for ANNR, which is similar to Theorem \ref{thm:convergence_annd_cm_strong}, with the exception of only needing a finite $(1 + \eta)$-moment.

\begin{theorem}[Convergence of ANNR in CM]\label{thm:annr_convergence_cm}
Let $\mathscr{D}$ be an integer-valued random variable which satisfies $\Exp{D^{1+\eta}} <
\infty$, for some $0 < \eta < 1$, and let $\{G_n\}_{n \ge 1}$ be a sequence of graphs generated by
CM with ${\bf D}_n=\emph{\texttt{IID}}(\mathscr{D})$. Then, for any $\delta < \frac{\eta}{8 + 4\eta}$, as $n \to \infty$,
\[
	\sup_{k \ge 0} \, \Prob{ \left|{\Theta_n(k) - \Exp{F^\ast(\mathscr{D}^\ast)}}\right|\ind{ f_n(k)>0}
    > n^{-\delta}}
	 = \bigO{n^{-\delta}}.
\]
\end{theorem}

To establish convergence of the ANNR in the Erased Configuration Model we follow a similar approach as for the ANND and study
the difference $\left|\widehat{\Theta}_n(k) - \Theta_n(k)\right|$. We will prove that this difference converges to zero,
in probability.

\begin{theorem}\label{thm:annr_erased_model_error_term}
Let $\mathscr{D}$ be regularly varying with exponent $\gamma > 1$ and let $\{G_n\}_{n \ge 1}$ be a sequence of graphs generated by
ECM with ${\bf D}_n=\emph{\texttt{IID}}(\mathscr{D})$. Then, for any $\delta > 0$ and $k$ such that $f^\ast(k) > 0$,
\[
	\lim_{n \to \infty} \Prob{\left|\widehat{\Theta}_n(k) - \Theta_n(k)\right|
	> \delta} = 0.
\]
\end{theorem}

With this result we obtain the following version of Theorem \ref{thm:annr_convergence_cm} for the ECM.

\begin{theorem}[Convergence of ANNR in ECM]\label{thm:convergence_annr_ecm_weak}
Let $\mathscr{D}$ be regularly varying with exponent $\gamma > 1$ and let $\{G_n\}_{n \ge 1}$ be a sequence of graphs generated by
ECM with ${\bf D}_n=\emph{\texttt{IID}}(\mathscr{D})$. Then, for any $\delta > 0$ and $k$ such that $f^\ast(k) > 0$,
\[
	\lim_{n \to \infty}  \Prob{ \left|{\widehat{\Theta}_n(k) - \Exp{F^\ast(\mathscr{D}^\ast)}}\right|\ind{ f_n(k)>0}
    > \delta} = 0.
\]
\end{theorem}

Similar to Theorem \ref{thm:convergence_annd_ecm_weak}, this result follows from Theorem \ref{thm:annr_convergence_cm}
and \ref{thm:annr_erased_model_error_term} by splitting
\[
	\left|{\widehat{\Theta}_n(k) - \Exp{F^\ast(\mathscr{D}^\ast)}}\right| \le
	\left|{\Theta_n(k) - \Exp{F^\ast(\mathscr{D}^\ast)}}\right|
	+\left|\widehat{\Theta}_n(k) - \Theta_n(k)\right|.
\]
Note, however, that Theorem \ref{thm:convergence_annr_ecm_weak} is less strong than Theorem \ref{thm:convergence_annd_ecm_weak},
since we only have convergence in probability. Replacing the lower bound $\delta$ inside the probability with
$n^{-\delta}$ requires a better understanding of the scaling of the number of removed stubs of nodes, which falls outside the focus of this paper.
Still, we show in the next section that the numerical results for ANNR in ECM are as strong as in CM.

\section{Numerical experiments}\label{sec:experiments}

In this section we will illustrate our results for ANND and ANNR in CM of different sizes $n$ and regularly-varying degree distributions. For the sake of illustration, we take $\gamma=2.5$ (finite variance) and $\gamma=1.5$ (infinite variance). We generate 100 realizations of each graph with each $\gamma$.

Here we will show results for the multi-graph CM. We have also computed the ANND and the ANNR for ECM. However, for small values of $k$ the results for CM and ECM were very close, as predicted by Theorem~\ref{thm:convergence_annd_ecm_weak}, and for large values of $k$ the results were unstable and inconclusive because of the small sample size. Since our main point here is to demonstrate the advantages of the ANNR, we will omit the results for ECM and leave this topic for future research.

In order to generate power law degree sequences, let $X$ be a Pareto random variable with exponent $\gamma > 1$,
\begin{equation*}
\Prob{X > t} = t^{-\gamma} \quad \text{for all } t \ge 1.
\end{equation*}
Then $\mathscr{D} = \floor{X}$, which is integer-valued, satisfies
\begin{equation}\label{eq:experiments_floor_pareto_distribution}
\Prob{\mathscr{D} = k} = \Prob{k \le X < k + 1} = k^{-\gamma} - (k + 1)^{-\gamma} \sim \gamma t^{-\gamma - 1}.
\end{equation}
In particular, it follows that $\mathscr{D} = \floor{X}$ satisfies \eqref{eq:def_regularly_varying_pmf} with eventually monotone slowly-varying function. Hence it is regularly varying with exponent $\gamma$.

We can express the first and second moment of $\mathscr{D}$ through the Riemann zeta function $\zeta$ as follows:
\[
\nu_1 = \zeta(\gamma) \quad \text{and} \quad
\nu_2 = 2\zeta(\gamma-1)-\zeta(\gamma),
\]
from which we obtain that $\nu_2/\nu_1 =2\zeta(\gamma-1)/\zeta(\gamma) - 1$. For specific numerical values see Table \ref{tbl:annr_limit_values}.

Now, for each value of $\gamma$ and $n$, we generate $100$ graphs $G_n$, using CM with ${\bf D}_n=\texttt{IID}(\mathscr{D})$. Then, for each $k$, we compute the average of $\Phi_n(k)$ and $\Theta_n(k)$, with respect to the $100$ generated graphs and plot these as a function of $k$.

\subsection{Computing $\Exp{F^\ast(\mathscr{D}^\ast)}$}

In order to validate our results for the ANNR in CM we need to know the value of its limit $\Exp{F^\ast(\mathscr{D}^\ast)}$.
Elementary computations yield
\[
F^\ast(k) = \frac{1}{\zeta(\gamma)}\sum_{t = 1}^k t^{-\gamma} - \frac{k(k+1)^{-\gamma}}{\zeta(\gamma)}.
\]
Now, since $\mathscr{D}^\ast$ has probability density function $f^\ast(k) = kf(k)/\zeta(\gamma)$ we have that
\begin{align*}
\Exp{F^\ast(\mathscr{D}^\ast)} &= \sum_{k = 1}^\infty F^\ast(k)f^\ast(k)\\
&= \frac{1}{\zeta(\gamma)^2 }\sum_{k = 1}^\infty \left(kf(k) \sum_{t = 1}^k t^{-\gamma}
- k^2(k+1)^{-\gamma} f(k)\right)\\
&= \frac{1}{\zeta(\gamma)^2 }\sum_{k = 1}^\infty \left( \left(k^{1-\gamma} - k(k + 1)^{-\gamma}\right)
\sum_{t = 1}^k t^{-\gamma} - k^{2-\gamma}(k+1)^{-\gamma} + k^2(k+1)^{-2\gamma}\right).
\end{align*}
The above expression cannot be written in terms of known functions. Hence, we need to evaluate it numerically.
To this end, we chose $N$ large enough such that the first two digits of $\sum_{k = 1}^M F^\ast(k)f^\ast(k)$
remain the same for $M \ge N$. We remark that for $\gamma = 1.5$ we had to take $N = 10^7$ while we could not achieve the precision up to the third digit even for $N = 5 \times 10^8$. The results are shown in Table \ref{tbl:annr_limit_values}.
\begin{table}[!h]
	\centering
	\begin{tabular}{lrr}
		$\gamma$ & $\nu_2/\nu_1$ & $\Exp{F^\ast(\mathscr{D}^\ast)}$ \\
		\hline
		$1.5$ & - & $0.545542$ \\
		$1.8$ & - & $0.592175$ \\
		$2.0$ & - & $0.625316$ \\
		$2.2$ & $6.502744$ & $0.658512$ \\
		$2.5$ & $2.894745$ & $0.706477$ \\
	\end{tabular}
	\caption{Numerical values of the limit of $\Phi_n$ and $\Theta_n$ as given by, respectively, Theorem \ref{thm:convergence_annd_cm_strong} and
		Theorem \ref{thm:annr_convergence_general}, for $\mathscr{D}$ with distribution \eqref{eq:experiments_floor_pareto_distribution} and different values of $\gamma$.}
	\label{tbl:annr_limit_values}
\end{table}

\subsection{ANND and ANNR}
We compute the ANND and the ANNR for 100 realizations of the sequence ${\bf D}_n=\texttt{IID}({\mathscr{D}})$. Recall that when there is no degree $k$ in ${\bf D}_n$, then $\Phi_n(k)$ and $\Theta_n(k)$ are set to zero. The values of $\Phi_n(k)$ and $\Theta_n(k)$ (including the zero values) are averaged over 100 CM graphs with ${\bf D}_n=\texttt{IID}({\mathscr{D}})$. Figures \ref{fig:annd_cm} and~\ref{fig:annr_cm} show the results for $\Phi_n$ and $\Theta_n$, respectively.
\begin{figure}[!pb]
	\centering
	\begin{subfigure}{0.48\linewidth}
		\includegraphics[scale=0.33]{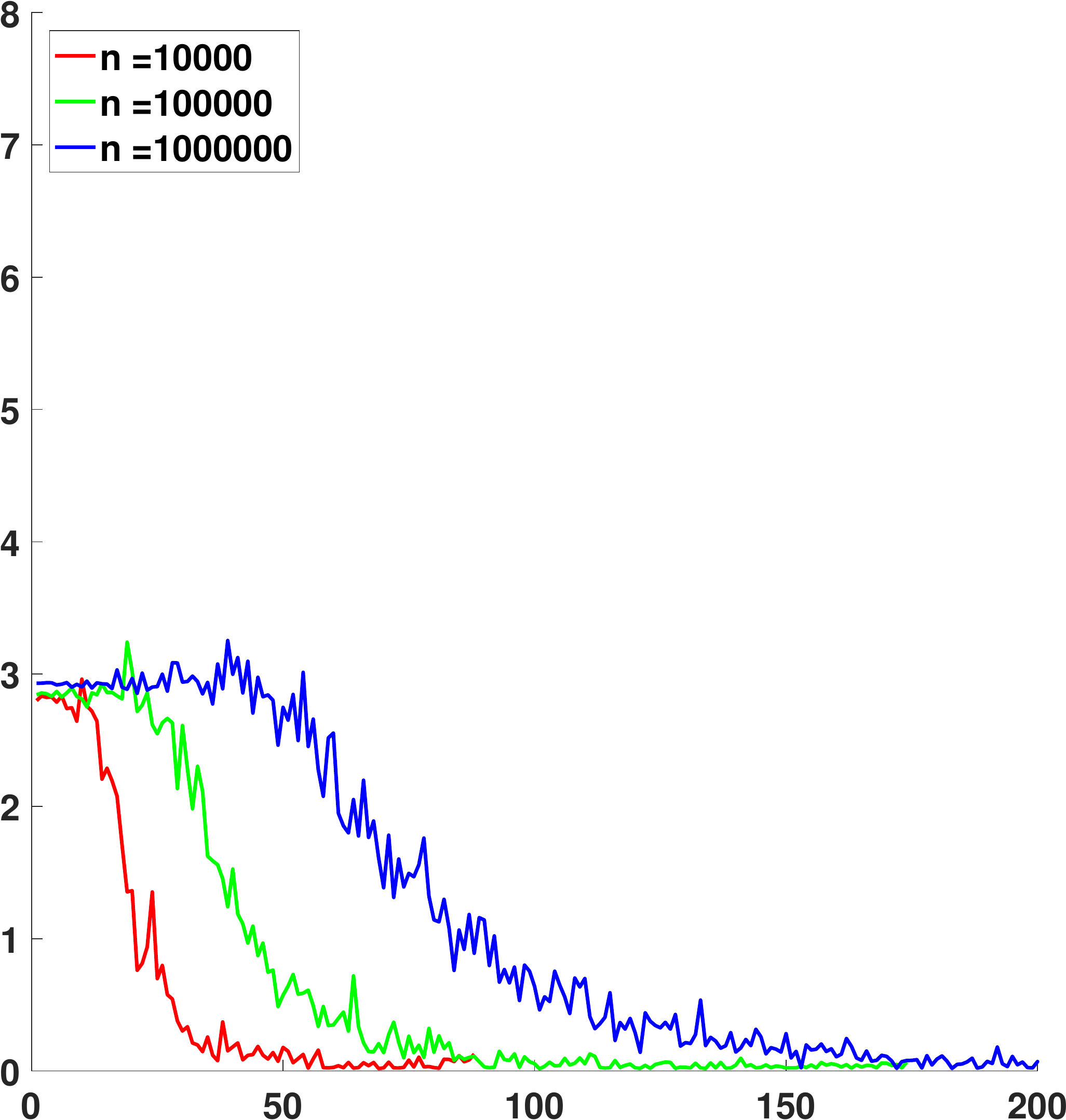}
		\caption{$\gamma = 2.5$}
	\end{subfigure}
	\begin{subfigure}{0.48\linewidth}
		\includegraphics[scale=0.33]{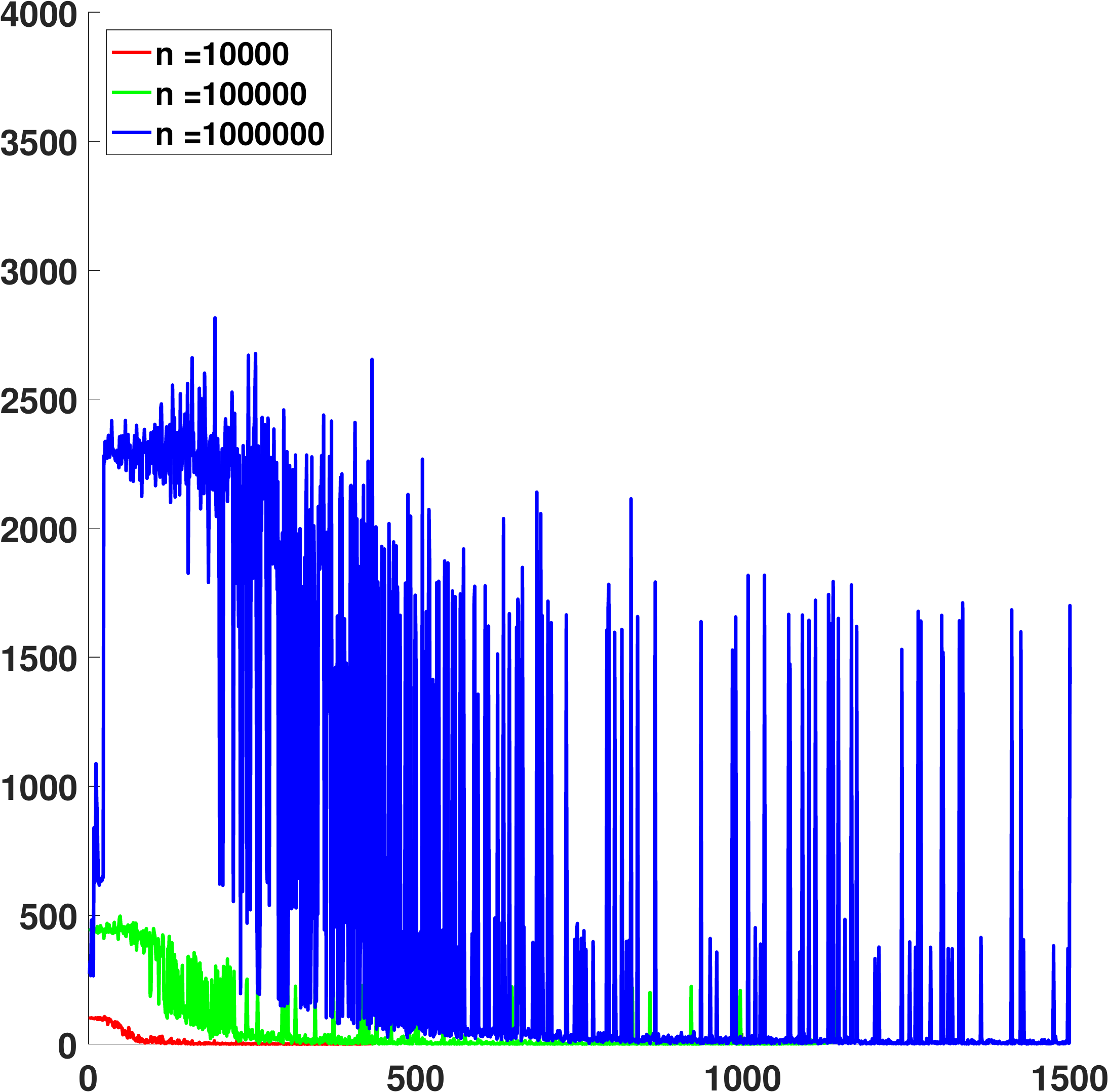}
		\caption{$\gamma = 1.5$}
	\end{subfigure}
	\caption{Plots for $\Phi_n(k)$, as a function of $k$, based on $100$ CM graphs with ${\bf D}_n = \texttt{IID}(\mathscr{D})$.}
	\label{fig:annd_cm}
\end{figure}
\begin{figure}[!pt]
	\centering
	\begin{subfigure}{0.48\linewidth}
		\includegraphics[scale=0.33]{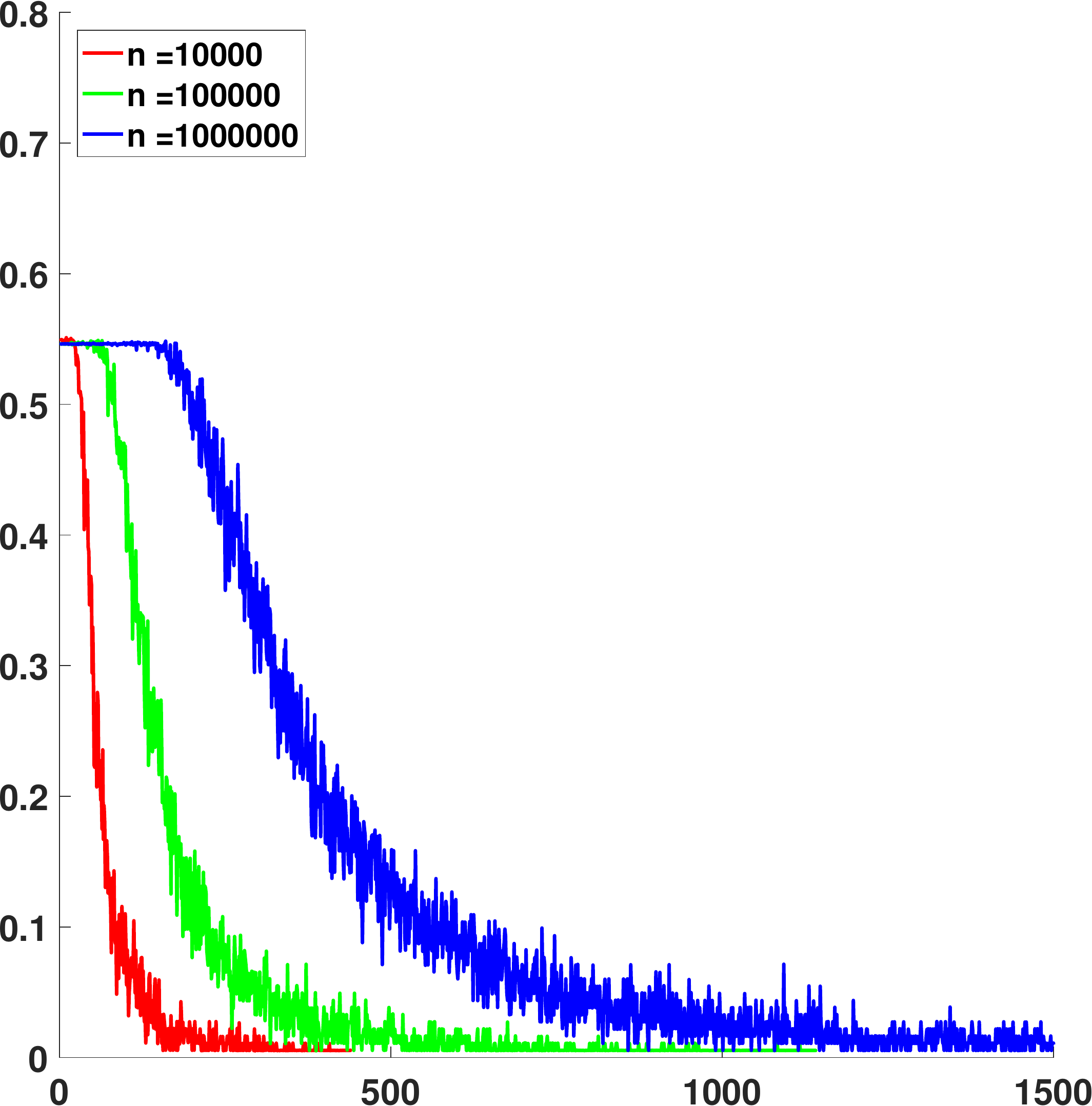}
		\caption{$\gamma = 1.5$}
	\end{subfigure}
	\begin{subfigure}{0.48\linewidth}
		\includegraphics[scale=0.33]{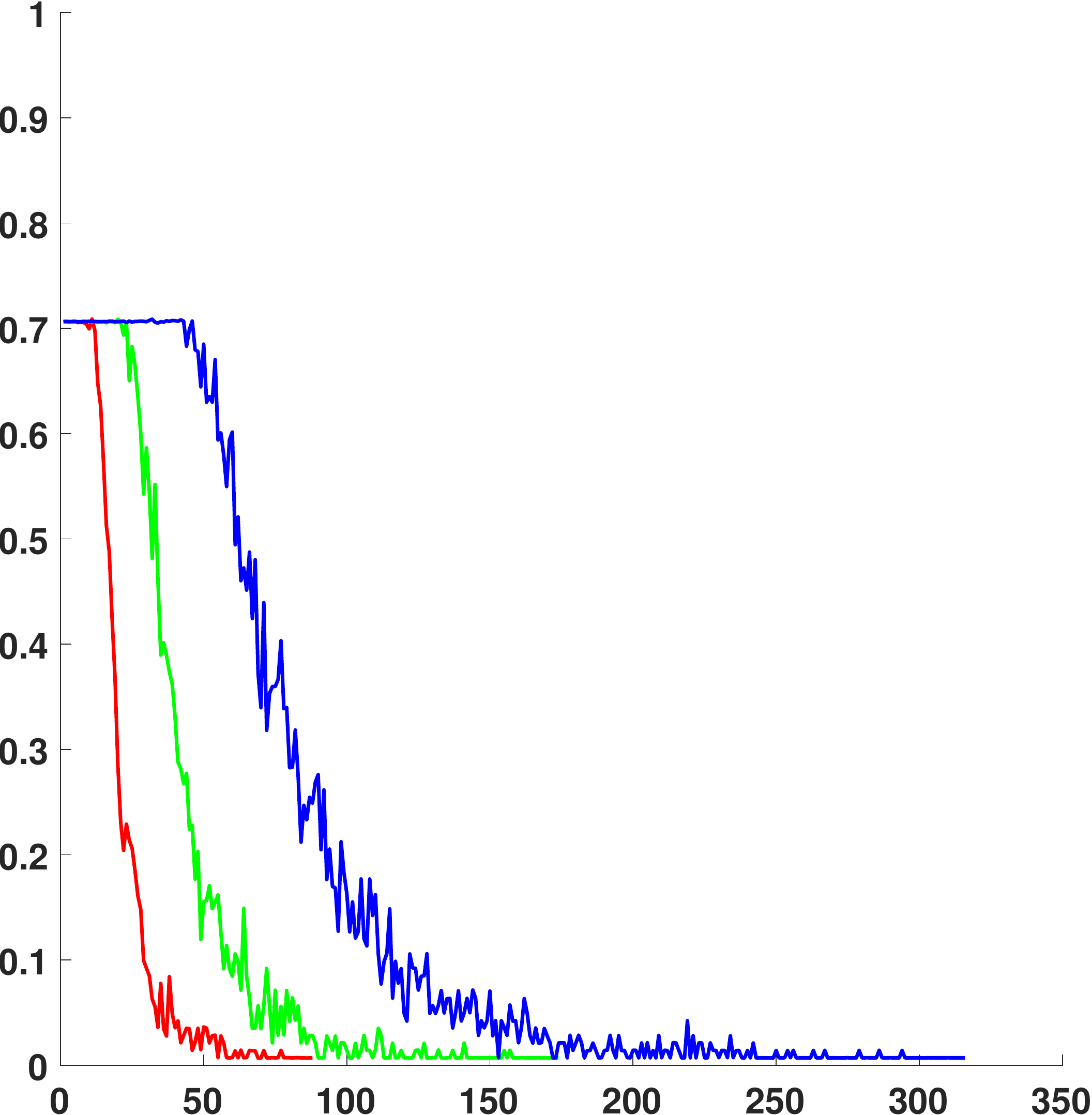}
		\caption{$\gamma = 2.5$}
	\end{subfigure}
	\caption{Plots for $\Theta_n(k)$, as a function of $k$, based on $100$ CM graphs with ${\bf D}_n = \texttt{IID}(\mathscr{D})$.}
	\label{fig:annr_cm}
\end{figure}
\begin{figure}[!h]
	\centering
		\includegraphics[scale=0.33]{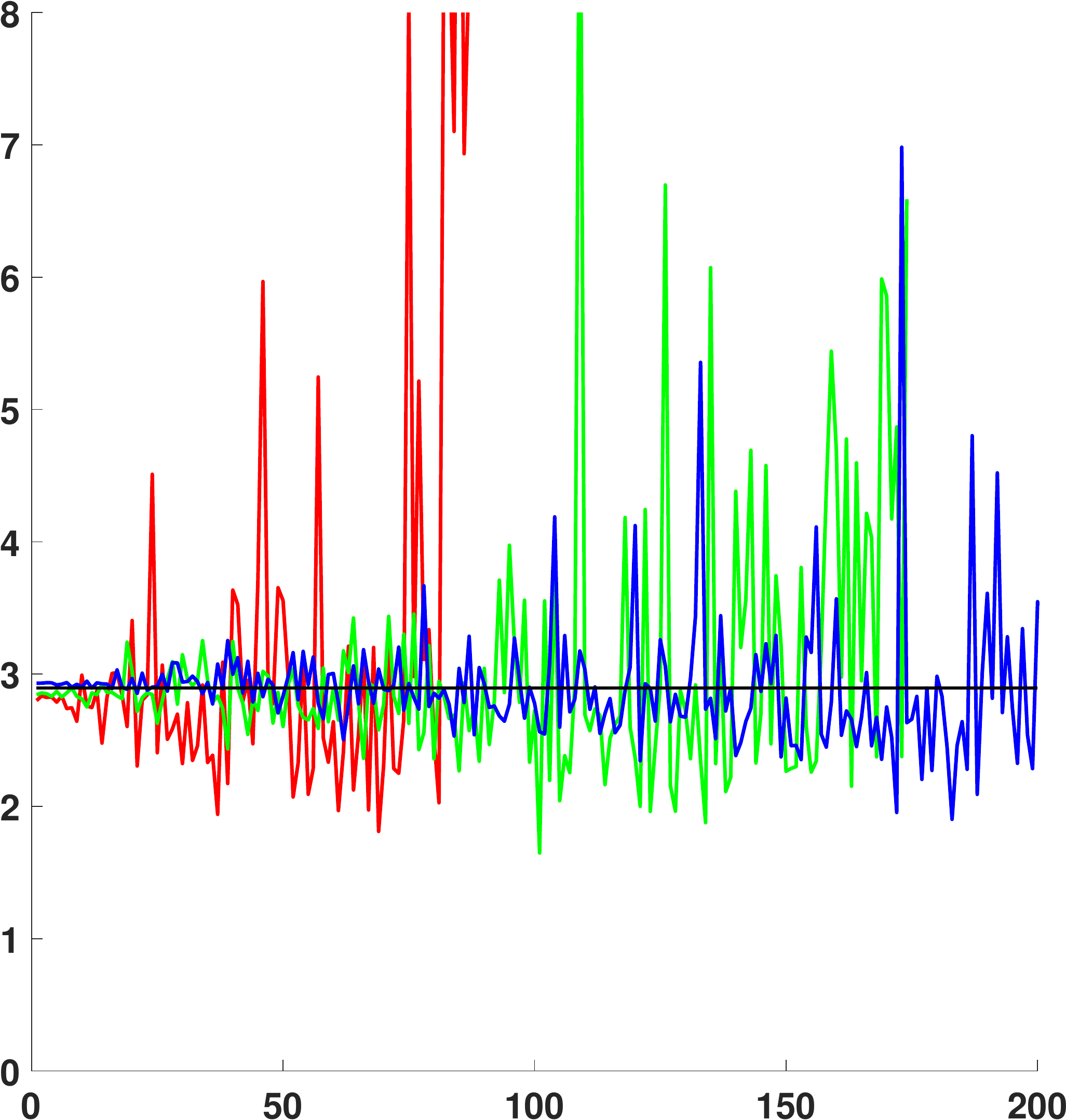}
	\caption{The corrected average of $\Phi_n(k)$, for $100$ graphs with $\gamma = 2.5$. The black line represents the theoretical limit $\nu_2/\nu_1$.}
	\label{fig:annd_cm_corrected}
\end{figure}

The first observation is that the ANND clearly diverges when $\gamma < 2$, while this does not happen for the ANNR.

The second observation is that both functions decline sharply when $k$ exceeds a certain threshold, which increases in $n$. This phenomenon is explained by Theorem \ref{thm:degree_sequence}. Indeed, according to this theorem, when $k$ is of a larger order of magnitude than $n^{1/(\gamma + 1)}$, then the probability that degree $k$ is present in ${\bf D}_n=\texttt{IID}({\mathscr{D}})$ is smaller than one and in fact converges to zero. In the numerical experiments, for large $k$, degree $k$ is not present in all 100 sampled sequences. Then $\Phi_n(k)$ and $\Theta_n(k)$ are set to zero, and this artificially decreases the average value of $\Phi_n(k)$ and $\Theta_n(k)$ over the 100 samples. Note that in Figures \ref{fig:annd_cm} and \ref{fig:annr_cm} the decline indeed starts around the point $k=n^{1/(\gamma+1)}$.

We can correct for this by counting, for each $k$, the number of graphs that have nodes of degree $k$ and then dividing by this number instead of dividing by the sample size $100$.
The result for the corrected ANND with $\gamma=2.5$ is shown in Figure \ref{fig:annd_cm_corrected}.
Observe that the plot no longer shows the decrease with $k$, and for small $k$ the value of $\Phi_n(k)$ is unchanged. However, for large $k$ we obtain very unstable fluctuations due to the small sample size. This can be possibly remedied, at least in the finite variance scenario, by averaging over a larger number of graphs. How large the sample size should be to reach stability for large $k$, and under which conditions it is possible, is another interesting question for future research.

 We have also obtained numerical results for the corrected $\widehat{\Phi}_n(k)$ in the ECM. These are omitted because the plots showed similar high fluctuations as in Figure~\ref{fig:annd_cm_corrected}, and numerical comparison between CM and ECM was inconclusive.
\begin{figure}[!p]
	\centering
	\begin{subfigure}{0.48\linewidth}
		\includegraphics[scale=0.33]{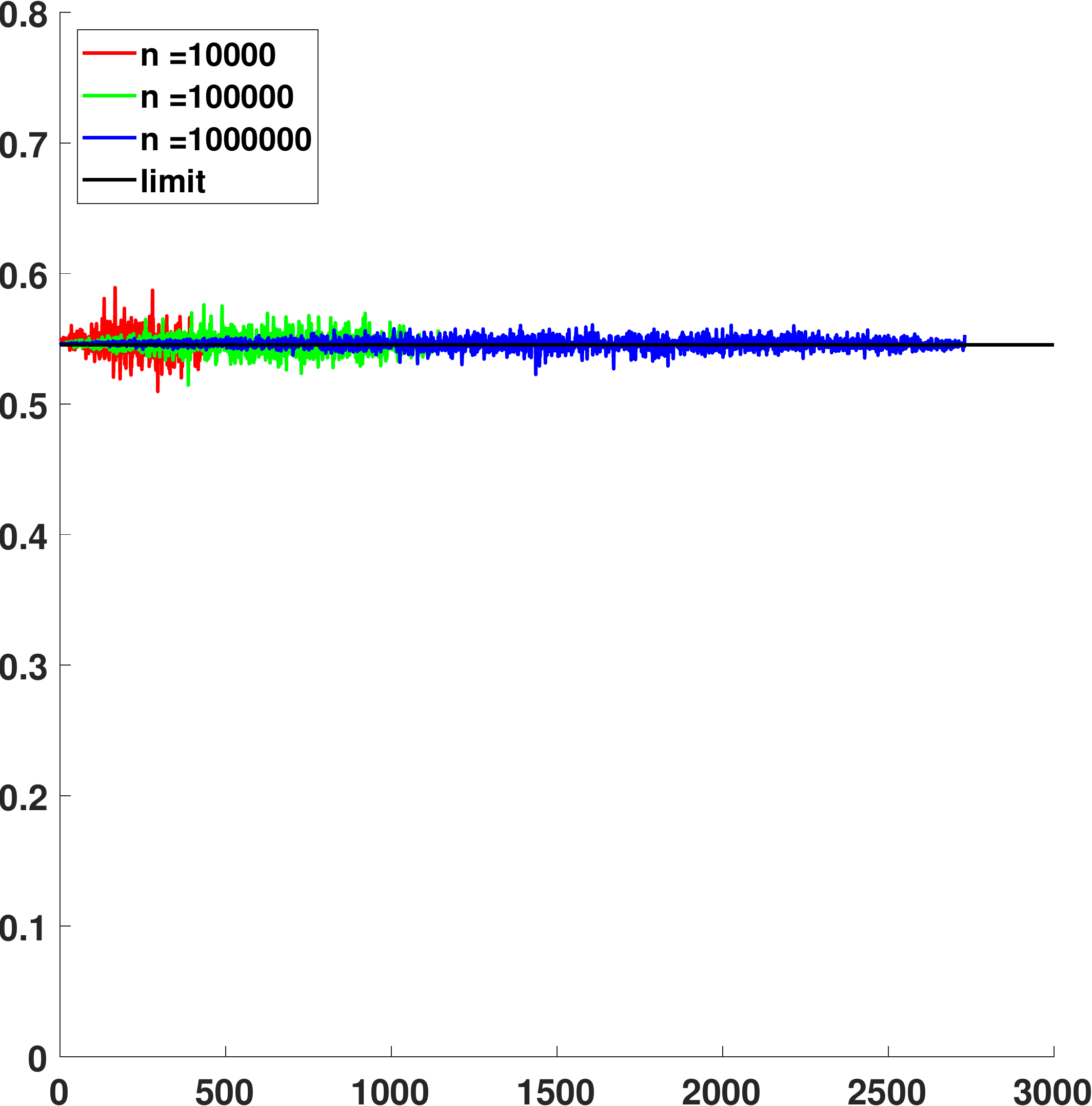}
		\caption{$\gamma = 1.5$}
		\label{sfig:annr_cm_15_corr_lim}
	\end{subfigure}
	\begin{subfigure}{0.48\linewidth}
		\includegraphics[scale=0.33]{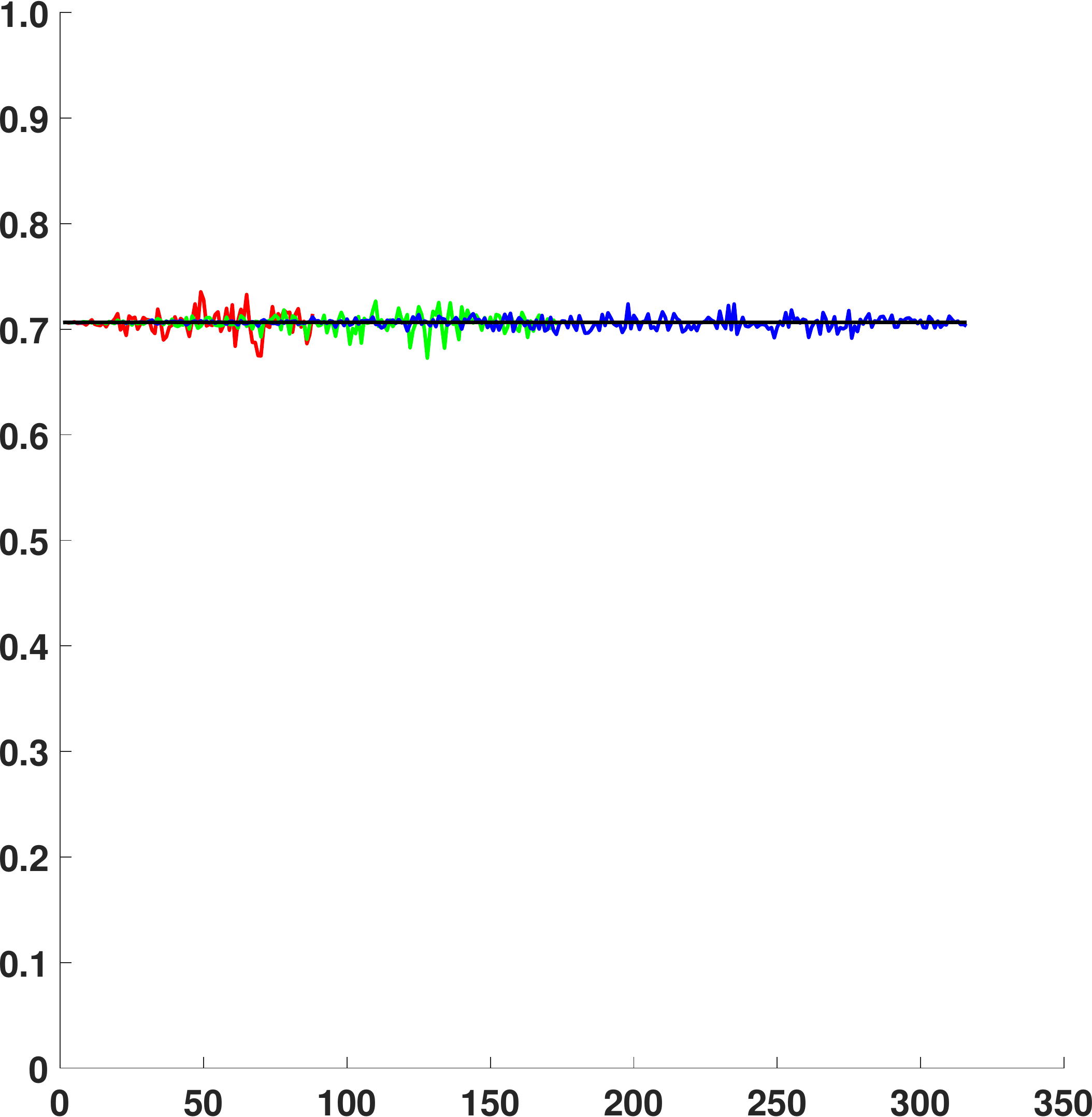}
		\caption{$\gamma = 2.5$}
		\label{sfig:annr_cm_25_coor_lim}
	\end{subfigure}
	\caption{Corrected average of $\Theta_n$ in CM, for $100$ graphs with $\gamma = 2.5$ and $\gamma = 1.5$. The black line represents the
		theoretical limit $\Exp{F^\ast(\mathscr{D}^\ast)}$.}
	\label{fig:annr_cm_corrected}
\end{figure}~
\begin{figure}[!p]
	\centering
	\begin{subfigure}{0.48\linewidth}
		\includegraphics[scale=0.33]{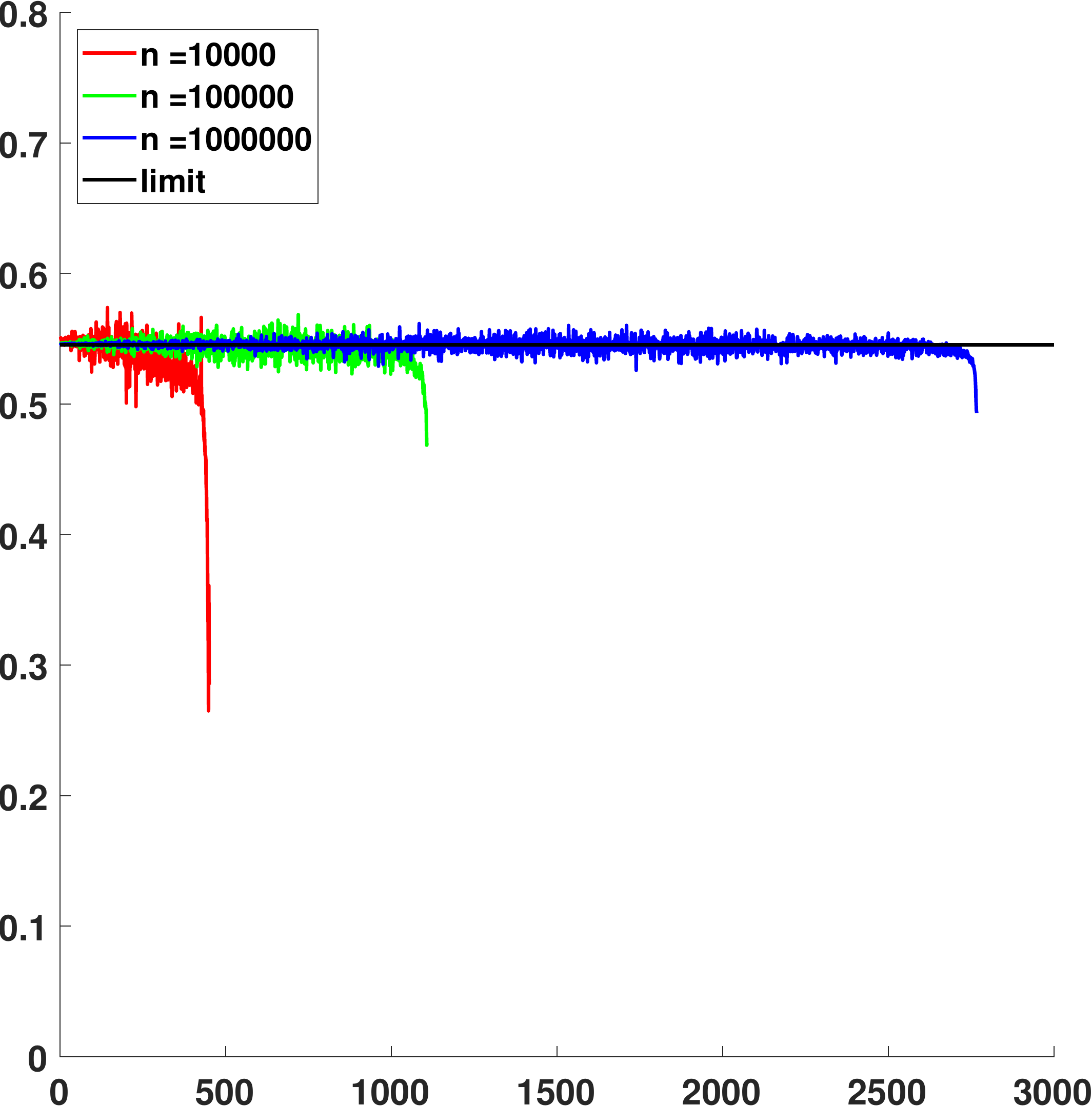}
		\caption{$\gamma = 1.5$}
		\label{sfig:annr_ecm_15_corr_lim}
	\end{subfigure}
	\begin{subfigure}{0.48\linewidth}
		\includegraphics[scale=0.33]{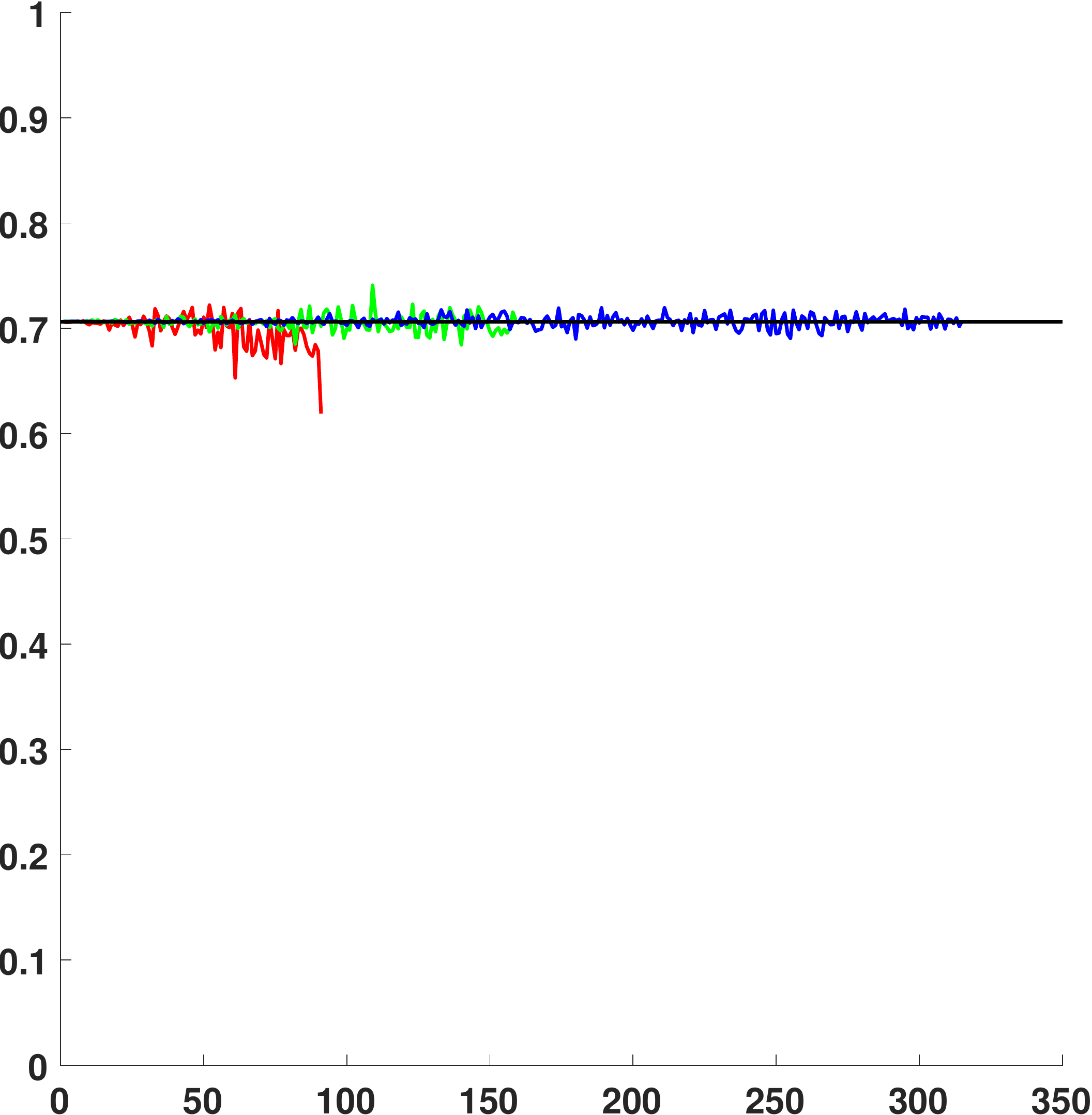}
		\caption{$\gamma = 2.5$}
		\label{sfig:annr_ecm_25_coor_lim}
	\end{subfigure}
	\caption{Corrected average of $\Theta_n$ in ECM, for $100$ graphs with $\gamma = 2.5$ and $\gamma = 1.5$. The black line represents the
		theoretical limit $\Exp{F^\ast(\mathscr{D}^\ast)}$.}
	\label{fig:annr_ecm_corrected}
\end{figure}

Since $\Theta_n(k) \le 1$ for all $k$ and $n$, ANNR remains stable after correction even in the infinite variance scenario. This is clearly observed in Figure \ref{fig:annr_cm_corrected} where we present the corrected plot of $\Theta_n(k)$ in CM, for $\gamma=2.5$ and $\gamma=1.5$. In both cases the plot very closely follows the straight line, which represents the limit value $\Exp{F^\ast(\mathscr{D})}$.

Finally, in Figure~\ref{fig:annr_ecm_corrected} we plot the corrected version of ANNR for the ECM. We see that the ANNR in ECM again shows a great stability, and we clearly observe the point-wise convergence to the right constant for each fixed $k$, as stated in Theorems~\ref{thm:annr_convergence_cm} and ~\ref{thm:annr_erased_model_error_term}. We also clearly observe structural correlations, or finite-size effects, for top values of the degrees, especially for $\gamma = 1.5$. Since the rank correlations are not affected by high dispersion in the values of the degrees, these finite-size effects can only be explained by simplicity of the graph. This is in agreement with previous work~\cite{Hoorn2015PhysRev}, where we observed structural correlations in another rank-based dependency measure --  Spearman's rho. We see that for ANNR, these effects appear only when $k$ is very large, say, greater than some $k_{critical}(n)$. One can expect that $k_{critical}(n)$ scales as a positive power of $n$, however, establishing the exact scaling for $k_{critical}(n)$ seems to be a difficult problem.

\section{Discussion}

The most important implication of our results is that the ANND cannot be used in the case when degrees have infinite variance. This is a very similar situation as was observed before for Pearson's correlation coefficient~\cite{Hofstad2014}. In particular, the ANND scales with the graph size, which makes it not suitable for comparison of networks of the same structure but different sizes. Moreover, even when rescaled, the ANND converges to a random variable instead of a number, and therefore it is impossible to establish any meaningful relation between the scaled ANND and the network's assortativity. Therefore, the use of the ANNR is strongly recommended.

In addition, we would like to mention two interesting open problems, stemming from this research, that we will address in the near future. First of all, in fact, the ANND deals with `double sampling' as follows. 1) In order to create a graph of size $n$, the degree sequence  $\texttt{IID}(\mathscr{D})$ is sampled. As we proved in Theorem~\ref{thm:degree_sequence}, each specific sufficiently large value $k$ will appear in such sequence only with a small probability, therefore it will take some time to sample several sequences that include such degree. 2) Each node in the network samples its neighbors. This double sampling gives rise to vast fluctuations of the ANND for large values of $k$ even in the CM with finite variance of the degrees, despite of the convergence of the ANND to a deterministic (constant) limit, that has been proved in Theorem~\ref{thm:convergence_annd_cm_strong}. We observed this phenomenon in Figure~\ref{fig:annd_cm_corrected}. The magnitude of these fluctuations and the number of graphs necessary to obtain convergence of the ANND for all $k$ have not been addressed in this paper, and are the focus of future research.

Second, it has been observed in the literature~\cite{Barabasi2016,Catanzaro2005} that the ANND in the ECM decreases for large values of $k$. This phenomenon has been ascribed to the simplicity of the graph, because large nodes do not have other large nodes to connect to, so they must create disassortative connecitons.  These `structural negative correlations', or `finite-size effects' are broadly recognized in the literature, see~\cite{Hoorn2016b} and references therein. Interestingly, our results state that ANND in the CM and the ECM are asymptotically equivalent for any fixed $k$, and the same holds for the ANNR. However,  numerical results suggest that ANND and even ANNR are subject to the finite-size effects for $k$ that scales as a power of $n$. The exact mathematical characterization of the finite-size effects in terms of ANND and ANNR remains an interesting open problem.

\section{Proofs}\label{sec:proofs}

In this section, we give the proofs of the results in the paper.

\subsection{Regularly-varying degrees}

We start with a small result, relating regular variation of $\Prob{X = k}$, for some integer-valued random variable, to that of the inverse cdf $\Prob{X > t}$.

\begin{lemma}\label{lem:regularly_varying_cdf}
Suppose that $X$ is an integer-valued random variable with probability function
\[
	\Prob{X = k} = l(k) k^{-\gamma-1}, \quad k = 1,2,\dots,
\]
with $\gamma > 1$ and some slowly-varying function $l(x)$ that is eventually monotone. Then, as $t \to \infty$,
\[
	\Prob{X > t} \sim \frac{l(t)}{\gamma} t^{-\gamma}.
\]
In particular this implies that
\[
	\Prob{X > t} = \tilde{l}(t) t^{-\gamma},
\]
with $\tilde{l}(x) \sim l(x)/\gamma$.
\end{lemma}

\begin{proof}
To prove the result we will bound the sums $\sum_{k > t}^\infty \Prob{X = k}$ by integrals, using that $l(x)$ is eventually monotone. We will assume that $l(x)$ is eventually monotonic decreasing. The proof for monotonic increasing $l(x)$ is similar.

First observe that
\[
	\Prob{X > t} \sim \frac{l(t)}{\gamma} t^{-\gamma},
\]
implies that $\tilde{l}(x) = x^\gamma \Prob{X > x}$ is slowly varying and $\tilde{l}(x) \sim l(x)/\gamma$, which proves the second statement.

For the first statement, let $K$ be the smallest integer such that for all $y \ge x \ge K$, $l(x) \ge l(y)$. In addition define for all $x \in [1,\infty)$, $f(x) = l(x)x^{-\gamma-1}$. Then for all $t \ge K + 1$
\begin{align*}
	\sum_{k = \ceil{t}}^\infty \Prob{X = k}
	&= \Prob{X = \ceil{t}} + \sum_{k = \ceil{t} + 1} \int_{k-1}^k \Prob{X = k} \, dx \\
	&\le \Prob{X = \ceil{t}} + \sum_{k = \ceil{t} + 1} \int_{k-1}^k f(x) \, dx \\
	&= \Prob{X = \ceil{t}} + \int_{\ceil{t}}^\infty f(x) \, dx.
\end{align*}
Similarly, we have for all $t \ge K + 1$,
\[
	\sum_{k = \ceil{t}}^\infty \Prob{X = k} = \sum_{k = \ceil{t}}^\infty \int_{k}^{k+1} \Prob{X = k} \, dx \ge \int_{\ceil{t}}^\infty f(x) \, dx.
\]
We therefore obtain
\begin{equation}\label{eq:bounds_cdf_integral}
	1 \le \frac{\sum_{k = \ceil{t}}^\infty \Prob{X = k}}{\int_{\ceil{t}}^\infty f(x) \, dx} 
	\le 1 + \frac{\Prob{X = \ceil{t}}}{\int_{\ceil{t}}^\infty f(x) \, dx}.
\end{equation}
By Karamata's theorem (\cite[1.5.11]{Bingham1989}) it follows that, as $t \to \infty$.
\begin{equation}\label{eq:karamata}
	\int_t^\infty f(x) \, dx \sim \frac{l(t) t^{-\gamma}}{\gamma}.
\end{equation}
Since $\Prob{X = \ceil{t}} = o\left(l(t)t^{-\gamma}\right)$, it follows from \eqref{eq:bounds_cdf_integral}
and \eqref{eq:karamata} that, as $t \to \infty$,
\begin{align*}
	\frac{\gamma\Prob{X > t}}{l(t)t^{-\gamma}} &= \left(\frac{\gamma \int_{\ceil{t}}^\infty f(x) \, dx}{l(t) 	
		t^{-\gamma}}\right)\left(\frac{\sum_{k = \ceil{t}}^\infty \Prob{X = k}}{ \int_{\ceil{t}}^\infty f(x) \, dx}\right)
	\sim 1,
\end{align*}
which finishes the proof.
\end{proof}

\subsection{Proof of Theorem \ref{thm:convergence_degree_densities}}\label{ssec:proof_convergence_degree_densities}

The proof we present here is an adjustment of \cite[Theorem 3.1]{Hoorn2016b} to the setting with
finite variance. It uses the following technical lemma, which is a direct consequence of the Burkholder's
inequality. See for instance \cite{Hoorn2016b} for a short proof.

\begin{lemma}[Lemma 3.5 \cite{Hoorn2016b}]\label{lem:burkholder}
Let $\{X_i\}_{i \ge 1}$ be a sequence of i.i.d. zero mean random variables such that $\Exp{|X_1|^{p}} < \infty$,
for some $1 < p < 2$. Then there exists a constant $K > 0$, which only depends on $p$, such that
\[
	\Exp{\left|\sum_{i = 1}^n X_i\right|^{p}\,} \le K n \Exp{|X_1|^{p}}.
\]
\end{lemma}

\begin{proof}[Proof of Theorem \ref{thm:convergence_degree_densities}]
Let $\varepsilon \le \eta/(4 + 2\eta)$, define the events
\[
	A_n = \left\{d_1(f_n,f) \le n^{-\varepsilon}\right\}, \quad B_n = \left\{d_1(f_n^\ast, f^\ast) \le n^{-\varepsilon}\right\}
\]
and note that $\Prob{\Omega_n^c} \le \Prob{A_n^c} + \Prob{A_n \cap B_n^c}$. Therefore it is enough to show that
\[
	\Prob{A_n^c} + \Prob{A_n \cap B_n^c} = O\left(n^{-\varepsilon}\right),
\]
as $n \to \infty$.

Since $\Exp{\mathscr{D}^{1 + \eta}} < \infty$ we use
\cite[Proposition 4]{Chen2015}, with $\alpha = 1 + \eta$, to obtain
\[
	\Prob{A_n^c} \le n^{\varepsilon} \Exp{d_1(f_n,f)} \le n^{\varepsilon - 1 + \frac{1}{1+\eta}}\left(\frac{2(1 + \eta)}{\eta}
	+ \frac{2}{1 - \eta}\right) \Exp{\mathscr{D}^{1 + \eta}} = O\left(n^{-\varepsilon}\right),
\]
where the last part follows from the fact that $\varepsilon - 1 + 1/(1 + \eta) < -\varepsilon$.

For the remaining probability, define
\[
	X_{ik} = D_i \ind{D_i > k} - \Exp{\mathscr{D} \ind{\mathscr{D} > k}},
\]
and write
\begin{align*}
	d_1(f_n^\ast, f^\ast) &= \sum_{k = 0}^\infty \left| \frac{1}{L_n}\sum_{i = 1}^n D_i \ind{D_i > k} - \Exp{\mathscr{D} \ind{\mathscr{D} > k}}\right|\\
	&\le \frac{1}{\nu n} \sum_{k = 0}^\infty \left|\sum_{i = 1}^n X_{ik}\right| + \left|\frac{1}{L_n} - \frac{1}{\nu n}\right| \sum_{k = 0}^\infty D_i \ind{D_i \ge k} \\
	&\le \frac{1}{\nu n} \sum_{k = 0}^\infty \left|\sum_{i = 1}^n X_{ik}\right| + \frac{|L_n - \nu n|}{\nu n}.
\end{align*}
Since on the event $A_n$ we have $|L_n - \nu n| \le n^{1 - \varepsilon}$ we get
\[
	\Prob{\frac{|L_n - \nu n|}{\nu n} > \frac{2n^{-\varepsilon}}{\nu + 1}, A_n} = \Prob{\frac{n^{-\varepsilon}}{\nu} > \frac{2n^{-\varepsilon}}{\nu + 1}, A_n} = 0,
\]
and hence
\begin{align*}
	\Prob{A_n \cap B_n^c} &= \Prob{d_1(f_n^\ast, f^\ast) > n^{-\varepsilon}, A_n} \\
	&\le \Prob{\frac{1}{\nu n} \sum_{k = 0}^\infty \left|\sum_{i = 1}^n X_{ik}\right| > \frac{(\nu-1) n^{-\varepsilon}}{\nu + 1}}
	+ \Prob{\frac{|L_n - \nu n|}{\nu n} > \frac{2n^{-\varepsilon}}{\nu + 1}, A_n} \\
	&= \Prob{\frac{1}{\nu n} \sum_{k = 0}^\infty \left|\sum_{i = 1}^n X_{ik}\right| > \frac{(\nu-1) n^{-\varepsilon}}{\nu + 1}}.
\end{align*}
For the last term let $p = 1/(1 - 2\varepsilon)$ and note that since $0 < \eta < 1$ and $0 < \varepsilon \le \eta/(2\eta + 4)$ we have
$1 < p \le 1 + \eta/2$. Therefore, by first applying Markov's inequality, followed by
H\"{o}lder's inequality and then Lemma \ref{lem:burkholder}, we get for some $K > 0$ depending only on $\varepsilon$,
\begin{align*}
	\Prob{\frac{1}{\nu n} \sum_{k = 0}^\infty \left|\sum_{i = 1}^n X_{ik}\right| > \frac{\nu n^{-\varepsilon}}{\nu + 1}}
	&\le \frac{\nu + 1}{\nu^2 n^{1 - \varepsilon}}\sum_{k = 0}^\infty \Exp{\left|\sum_{i = 1}^n X_{ik}\right|} \\
	&\le \frac{\nu + 1}{\nu^2 n^{1 - \varepsilon}}\sum_{k = 0}^\infty \Exp{\left|\sum_{i = 1}^n X_{ik}\right|^p \,}^{\frac{1}{p}}\\
	&\le \frac{\nu + 1}{\nu^2 n^{1 - \varepsilon}}\sum_{k = 0}^\infty \left(K n\Exp{ \left|X_{1k}\right|^p \,}\right)^{\frac{1}{p}}\\
	&\le \frac{K^{\frac{1}{p}}n^{\frac{1}{p}}(\nu + 1)}{\nu^2 n^{1 - \varepsilon}}\sum_{k = 0}^\infty \Exp{\left|X_{1k}\right|^p \,}^{\frac{1}{p}}\\
	&\le \frac{2 K^{\frac{1}{p}}n^{\frac{1}{p}}(\nu + 1)}{\nu^2 n^{1 - \varepsilon}}\sum_{k = 0}^\infty \Exp{\mathscr{D}^p \ind{\mathscr{D} > k}}^{\frac{1}{p}}.
\end{align*}
To finish the argument we write
\begin{align*}
	\Exp{\mathscr{D}^p \ind{\mathscr{D} > k}} &= \Exp{\mathscr{D}^p \ind{\mathscr{D} > k}}\ind{k < 1}
		+ \Exp{\mathscr{D}^p \ind{\mathscr{D} > k}}\ind{k \ge 1}\\
	&\le \Exp{\mathscr{D}^{p + 1}}\ind{k < 1} + k^{-p}\Exp{\mathscr{D}^{2p}}\ind{k \ge 1},
\end{align*}
so that, using $\Gamma(s)$ to denote the Gamma function,
\begin{align*}
	\frac{2 K^{\frac{1}{p}}n^{\frac{1}{p}}(\nu + 1)}{\nu^2 n^{1 - \varepsilon}}\sum_{k = 0}^\infty \Exp{\mathscr{D}^p \ind{\mathscr{D} > k}}^{\frac{1}{p}}
	&\le \frac{2 K^{\frac{1}{p}}n^{\frac{1}{p}}(\nu + 1)}{\nu^2 n^{1 - \varepsilon}} \left(\Exp{\mathscr{D}^{p + 1}}
		+ \Exp{\mathscr{D}^{2p}} \sum_{k = 1}^\infty k^{-p}\right)\\
	&= n^{\frac{1}{p} + \varepsilon - 1} \frac{2 K^{\frac{1}{p}}(\nu + 1)}{\nu^2} \left(\Exp{\mathscr{D}^{p + 1}}
			+ \Exp{\mathscr{D}^{2p}} \Gamma(p)\right)\\
	&= O\left(n^{\frac{1}{p} + \varepsilon - 1}\right) = O\left(n^{-\varepsilon}\right),
\end{align*}
where we used that $\Exp{\mathscr{D}^{p + 1}} \le \Exp{\mathscr{D}^{2p}} \le \Exp{\mathscr{D}^{2 + \eta}} < \infty$
and $1/p + \varepsilon - 1 = -\varepsilon$.
\end{proof}

\subsection{ANND in general graphs}\label{ssec:proof_annd_general_random_graphs}
\begin{proof}[Proof of Theorem \ref{thm:convergence_annd}]
Let  $\Omega_n$ and $\Gamma_n$ be as defined in Assumptions \ref{asmp:regularity_degrees} and
\ref{asmp:regularity_structure}, respectively. Define $\Lambda_n = \Omega_n \cap \Gamma_n$. Then by the assumptions we have that
\begin{equation}
\label{eq:lim_prob_lambdan}
  \lim_{n\To \infty} \Prob{\Lambda_n} = 1,
\end{equation}
and hence it is enough to prove the result conditioned on the event $\Lambda_n$. For this we first split
$\left|\Phi(k)-{\Phi}_n(k)\right|$ into two terms as follows:
\begin{align*}
	\abs{\Phi(k)-{\Phi}_n(k)}1_{\{f_n(k)>0\}}
	&\leq \ind{f_n(k)>0}\abs{\frac{1}{f_n^*(k)}-\frac{1}{f^*(k)}}\sum_{\ell = 1}^{\infty}
	  h_n(k,\ell)\ell \\
	&\hspace{10pt}+ \ind{f^*(k)>0} \frac{1}{f^*(k)}
	  \abs{\sum_{\ell=1}^{\infty} h_n(k,\ell)\ell-h(k,\ell)\ell}:= \Xi_n^{(1)} + \Xi_n^{(2)}
\end{align*}

We obtain a bound for $\Xi_n^{(1)}$ on $\Lambda_n$ by bounding both multiplicative terms in $\Xi_n^{(1)}$. To this end, first, on $\Gamma_n$ we obtain
\begin{align*}
	\sum_{\ell = 1}^{\infty} h_n(k,\ell)\ell
	&\le \sum_{k, \ell = 1}^{\infty} h(k,\ell)\ell
		+ \abs{\sum_{k, \ell = 1}^{\infty} h_n(k,\ell)\ell - h(k,\ell)\ell} \\
	 &\le \sum_{\ell = 1}^{\infty} f^\ast(\ell)\ell + d_1(f_n^\ast, f^\ast)
	= \Exp{\mathscr{D}^2} + n^{-\varepsilon}. \numberthis \label{eq:convergence_annd_11}
\end{align*}
Next we see that on $\Omega_n$,
\begin{equation}\label{eq:convergence_annd_12}
	\abs{\frac{1}{f_n^*(k)}-\frac{1}{f^*(k)}} = \frac{\abs{ f_n^*(k)-f^*(k)}}{f_n^*(k) f^*(k)}
	\le \frac{d_1(f_n^\ast, f^\ast)}{f^\ast(k)^2 - f^\ast(k) n^{-\varepsilon}}
	\le \frac{n^{-\varepsilon}}{f^\ast(k)^2 - f^\ast(k) n^{-\varepsilon}}.
\end{equation}

Combining \eqref{eq:convergence_annd_11} and \eqref{eq:convergence_annd_12}, we obtain
\begin{equation}\label{eq:convergence_annd_bound_term1}
  \Xi_n^{(1)} \le \frac{n^{-\varepsilon}}{f^\ast(k)^2 - f^\ast(k) n^{-\varepsilon}}(\Exp{\mathscr{D}^2} + n^{-\varepsilon})= O(n^{-\varepsilon}).
\end{equation}

In order to bound $\Xi_n^{(2)}$, we use a cut-off technique. Choose $w_n=\lfloor n^{p} \rfloor $, where $p$ is a positive constant to be
determined. Then we write
\begin{align*}
  \Xi_n^{(2)}\le \frac{\ind{f_n(k)>0}}{ f^\ast(k)}\abs{\sum_{\ell=1}^{w_n}h_n(k,\ell)l-h(k,\ell)\ell} 
  + \frac{1}{f^\ast(k)}\abs{\sum_{\ell>w_n}h_n(k,\ell)\ell-h(k,\ell)\ell }
  := \Xi_n^{(3)} + \Xi_n^{(4)}.
\end{align*}
To control $\Xi_n^{(3)}$, we use that on $\Gamma_n$, $\abs{\sum_{\ell=1}^{w_n}h_n(k,\ell)-h(k,\ell)}
\leq n^{-\kappa}$, so that on $\Lambda_n$,
\begin{equation}\label{eq:convergence_annd_bound_21}
  \Xi_n^{(3)}
  \le \frac{w_n \abs{\sum_{\ell=1}^{w_n}h_n(k,\ell)-h(k,\ell)}}{f^\ast(k)}
  \le \frac{n^{-\kappa+p}}{f^\ast(k)}.
 \end{equation}

For $\Xi_n^{(4)}$, we use that on $\Omega_n$ we have
\begin{align*}
	\frac{1}{f^\ast(k)}\abs{\sum_{\ell > w_n} h_n(k,\ell)\ell - h(k,\ell)\ell}
	&\le \frac{w_n^{-\eta}}{f^\ast(k)} \left(
		\sum_{\ell > w_n} h_n(k,\ell)\ell^{1 + \eta} + h(k,\ell)\ell^{1 + \eta}\right).
\end{align*}
Now, we obtain
\[
	\sum_{\ell > w_n} h(k,\ell)\ell^{1 + \eta}
	\le \sum_{k, \ell > 0} h(k,\ell)\ell^{1 + \eta}
	= \sum_{\ell > 0} f^\ast(\ell) \ell^{1 + \eta}
	= \frac{1}{\nu_1} \sum_{\ell = 0}^{\infty} f(\ell)\ell^{2 + \eta}
	= \frac{\Exp{\mathscr{D}^{2 + \eta}}}{\nu_1}.
\]
Next, on $\Omega_n$ we have that
\begin{align*}
	\sum_{\ell > w_n} h_n(k,\ell)\ell^{1 + \eta}
	&\le \sum_{k, \ell > 0} h_n(k,\ell)\ell^{1 + \eta} = \sum_{\ell > 0} f_n^\ast(\ell)\ell^{1 + \eta} \\
	&= \frac{1}{L_n} \sum_{\ell > 0}\ell^{2 + \eta} \sum_{i = 1}^n \ind{D_i = \ell}\le \frac{1}{\nu_1 n - n^{1-\varepsilon}} \sum_{i = 1}^n D_i^{2 + \eta}.
\end{align*}
Therefore, using Markov's inequality, we have, for all $0<c<1$,
\begin{align*}
	\Prob{\Xi_n^{(4)} > c n^{-\delta}, \Lambda_n}
	&\le \frac{n^\delta w_n^{-\eta}}{c f^\ast(k)}\left(\frac{\Exp{\mathscr{D}^{2 + \eta}}}{\nu_1 - n^{-\varepsilon}} + \frac{\Exp{\mathscr{D}^{2 + \eta}}}{\nu_1}\right)\\
	&= \frac{n^{\delta} w_n^{-\eta}}{c f^\ast(k)(\nu_1 - n^{-\varepsilon})}(2\Exp{\mathscr{D}^{2 + \eta}}+n^{-\varepsilon})
	= O\left(n^{\delta - p\eta}\right). \numberthis \label{eq:convergence_annd_bound_22}
\end{align*}
We now use a standard bound to obtain
\begin{align*}
\Prob{\abs{\Phi(k) - \Phi_n(k) }> n^{-\delta},\Lambda_n}&\le \Prob{\Xi_n^{(1)}> \frac{1}{3}\, n^{-\delta},\Lambda_n}\\&+\Prob{\Xi_n^{(3)}> \frac{1}{3}\, n^{-\delta},\Lambda_n}+\Prob{\Xi_n^{(4)}> \frac{1}{3}\, n^{-\delta},\Lambda_n}.
\end{align*}
It follows from \eqref{eq:convergence_annd_bound_term1}, \eqref{eq:convergence_annd_bound_21} and \eqref{eq:convergence_annd_bound_22} that whenever $p<\kappa$ and
\begin{equation}\label{eq:convergence_annd_condition_delta}
  0 < \delta < \min\{\varepsilon,\kappa-p, p\eta\},
\end{equation}
we have
\begin{equation}
\label{eq:lim_phin_lambdan}
\lim_{n \to \infty} \Prob{\abs{\Phi(k) - \Phi_n(k) }> n^{-\delta},\Lambda_n} = 0.
\end{equation}
Finally, set $p = \kappa/(\eta + 1) > 0$ so that $\kappa-p = p\eta$. Then \eqref{eq:convergence_annd_condition_delta} holds for all $\delta <\min\left\{\varepsilon, \frac{\kappa \eta}{\eta+1}\right\}$, and the result follows from \eqref{eq:lim_phin_lambdan} and \eqref{eq:lim_prob_lambdan}.
\end{proof}

\subsection{ANND in the configuration model}
\label{sec:proofs_cm}

\begin{proof}[Proof of Theorem \ref{thm:convergence_annd_cm_strong}]
Let $\Omega_n$ be defined as in Assumption \ref{asmp:regularity_degrees} with $\varepsilon = \delta/2 \le
\eta/(4 + 2\eta)$. 

First, observe that
\begin{equation*}
\frac{\nu_2}{\nu_1} =\sum_{\ell=1}^{\infty}f^*(\ell)\ell\,.
\end{equation*}
Hence we get
\begin{align}
	\abs{\Phi_n(k)-\frac{\nu_2}{\nu_1}}1_{\{f_n(k)>0\}} &=\frac{1}{f_n^*(k)} \abs{\sum_{\ell>0=1}^{\infty}(h_n(k,\ell)-f_n^*(k)
		f^*(\ell))\ell}\ind{f_n(k)>0} \notag \\
	&\leq \frac{1}{f_n^*(k)} \abs{\sum_{\ell=1}^{\infty} (h_n(k,\ell)-f_n^*(k) f_n^*(\ell))\ell}
		\ind{f_n(k) > 0} \label{eq:cm_convergence_hn_vs_fn}\\
	&\hspace{10pt}+ \abs{\sum_{\ell=1}^{\infty} (f_n^*(\ell) -f^*(\ell))\ell}
		\label{eq:cm_convergence_fn_vs_f}.
\end{align}
Term \eqref{eq:cm_convergence_fn_vs_f} is independent of $k$ and satisfies
\[
	\abs{\sum_{\ell=1}^{\infty} (f_n^*(\ell) -f^*(\ell))\ell} \le d_1\left(f_n^\ast, f^\ast\right).
\]
Therefore we have that
\[
	\Prob{\abs{\sum_{\ell=1}^{\infty} (f_n^*(\ell) -f^*(\ell))\ell} > \frac{n^{-\delta}}{2}, \Omega_n} = O\left(n^{-\delta}\right).
\]

We are now left with \eqref{eq:cm_convergence_hn_vs_fn}, which requires a bit more work.
We will prove that there exists a constant $C$, independent of $n$ and $k$, such that
\begin{equation}\label{eq:convergence_annd_conditioned_omega}
	\Prob{\left|\Phi_n(k) - \frac{\nu_2}{\nu_1}\right|\ind{f_n(k) > 0} > n^{-\delta}, \Omega_n}
	\le C \left(n^{\delta - \frac{\eta}{2(\eta + 2)}}\right).
\end{equation}
This together with Theorem~\ref{thm:convergence_degree_densities} will give the result.

Recall that $G_{ij}$ is the number of edges from $i$ to $j$ where self-loops are counted twice. Let us now define
\begin{equation}\label{eq:convergence_cm_Xij_def}
  X_{ij}(k,\ell)=1_{\{D_i=k,D_j=\ell\}}\left(\frac{G_{ij}}{L_n} - \frac{D_iD_j}{{L_n}^2}\right)\,.
\end{equation}
Then we have that
\begin{align}
\abs{h_n(k,\ell)-f_n^\ast(k)f_n^\ast(\ell)} &=\abs{\sum_{i,j=1}^nX_{ij}(k,\ell)}. \label{eq:diff_hn_h_Xij}
\end{align}
Now we will again use a cut-off technique. Let $p = 1/(2(\eta + 2))$, denote $w_n = \floor{n^p}$ and use
\eqref{eq:diff_hn_h_Xij} to bound \eqref{eq:cm_convergence_hn_vs_fn} as follows:
\begin{align*}
	&\hspace{-30pt}\frac{1}{f_n^\ast(k)} \abs{\sum_{\ell=1}^{\infty} (h_n(k,\ell)-f_n^\ast(k)
		f_n^\ast(\ell))\ell} \ind{f_n(k) > 0}\\
	&\le \frac{\ind{f_n(k) > 0}}{f_n^\ast(k)} \sum_{\ell=1}^\infty
		\abs{\sum_{i, j = 1}^n \ell X_{ij}(k,\ell)} \\
	&\le \frac{\ind{f_n(k) > 0}}{f_n^\ast(k)} \sum_{\ell=1}^{w_n}
		\abs{\sum_{i, j = 1}^n \ell X_{ij}(k,\ell)}
		+ \frac{\ind{f_n(k) > 0}}{f_n^\ast(k)} \sum_{\ell > w_n}
		\sum_{i, j = 1}^n D_j \abs{X_{ij}(k,\ell)}\\
	&:= \Xi_n^{(1)} + \Xi_n^{(2)}.
\end{align*}
We will bound the probability $\Prob{\Xi_n^{i} > \frac{n^{-\delta}}{4},\Omega_n}$ for $i = 1, 2$, separately, starting with $\Xi_n^{(2)}$. 
Observe that, if
$\mathbb{E}_n$ denotes the conditional expectation given the degree sequence, then it follows that
\begin{equation*}
 \mathbb{E}_n(G_{ij})=
\left\{
\begin{array}{ll}
\frac{D_iD_j}{L_n-1},&\textrm{if $i\neq j$},\\
 \frac{D_i(D_i-1)}{L_n-1},&\textrm{if $i = j$},
 \end{array}
  \right.
\end{equation*}
and hence by \eqref{eq:convergence_cm_Xij_def} we get
\begin{equation}
\mathbb{E}_n(\abs{X_{ij}(k,\ell)}) \leq 4\,\frac{D_iD_j}{{L_n}^2}1_{\{D_i=k,D_j=\ell\}}\,.
\end{equation}
Therefore we obtain
\begin{align*}
	\Prob{\Xi_n^{(2)} > \frac{n^{-\delta}}{4}, \Omega_n}
	&=\Prob{\frac{\ind{f_n(k) > 0}}{f_n^*(k)} \sum_{\ell > w_n}
		\sum_{i, j = 1}^n D_j \abs{X_{ij}(k,\ell)} > \frac{n^{-\delta}}{4}, \Omega_n} \\
	&\le 4n^{\delta} \sum_{\ell > w_n} \CExp{\frac{1}{f_n^*(k)} \ind{f_n(k)>0}
		\abs{\sum_{i,j=1}^n X_{ij}(k,\ell)D_j}}{\Omega_n} \\
	&\le 16n^{\delta} \sum_{i,j=1}^n \CExp{\frac{1}{f_n^*(k)} \ind{f_n(k)>0}
		\frac{D_i{D_j}^2}{{L_n}^2} \ind{D_j>w_n} \ind{D_i=k}}{\Omega_n} \\
	&= 16n^{\delta}\sum_{j=1}^{n} \CExp{\frac{D_j^2}{L_n}\ind{D_j >w_n}}{\Omega_n}\\
	&= 16n^{\delta+1} \CExp{\frac{D_1^2}{L_n}\ind{D_1>w_n}}{\Omega_n}  \\
	&\leq \frac{16 n^{\delta}}{\nu_1 - n^{-\varepsilon}} \Exp{D^2 \ind{D>w_n}}\\
	&\leq \frac{16}{\nu_1 - n^{-\varepsilon}} \Exp{D^{2+\eta}}n^{\delta}w_n^{-\eta}.
\end{align*}
From this it follows that there exists a constant $C_1 > 0$, independent of $k$,
such that
\begin{equation}\label{eq:cm_convergence_annd_bound_ell_ge_wn}
	\Prob{\Xi_{n}^{{2}} > \frac{n^{-\delta}}{4}, \Omega_n} \le C_1 n^{\delta - p \eta} = C_1 n^{\delta - \frac{\eta}{2(\eta + 2)}}.
\end{equation}
Finally we deal with $\Xi_n^{(1)}$. Using Markov's inequality and Cauchy-Schwartz
inequality, we get
\begin{align*}
	\Prob{\Xi_n^{(1)} > \frac{n^{-\delta}}{4}, \Omega_n}
	&= \Prob{\frac{\ind{f_n(k) > 0}}{f_n^\ast(k)} \sum_{\ell=1}^{w_n}
		\abs{\sum_{i, j = 1}^n \ell X_{ij}(k,\ell)} > \frac{n^{-\delta}}{4}, \Omega_n}\\
	&\le 4n^{\delta} \sum_{\ell=1}^{w_n}  \CExp{\frac{1}{f_n^\ast(k)} \ind{f_n(k)>0}
		\abs{\sum_{i,j=1}^n X_{ij}(k,\ell)\ell}}{\Omega_n} \\
	&= 4n^{\delta} \sum_{\ell=1}^{w_n} \ell \, \CExp{\frac{1}{f_n^\ast(k)} \ind{f_n(k)>0}
		\Expn{ \abs{\sum_{i,j=1}^n X_{ij}(k,\ell)}}}{\Omega_n}\\
	&\le 4n^{\delta} \sum_{\ell=1}^{w_n} \ell \, \CExp{\frac{1}{f_n^*(k)} \ind{f_n(k)>0}
		\Expn{\abs{\sum_{i,j=1}^n X_{ij}(k,\ell)}^2}^{1/2}}{\Omega_n}
\end{align*}
In order to bound  $\Expn{\abs{\sum_{i,j=1}^n X_{ij}(k,\ell)}^2}^{1/2}$ we use Lemma~\ref{lem:variance_Xij_kl} and, in particular, its Corollary \ref{cor:variance_Xij_kl} below.
Although Lemma \ref{lem:variance_Xij_kl} is crucial for the current proof, its own proof is quite technical. Hence, we will postpone it till the end of this section.

Invoking the result from Corollary~\ref{cor:variance_Xij_kl} we get a constant $C_2$,
independent of $k$ such that
\begin{align*}
	&\hspace{-30pt}n^{\delta} \sum_{\ell=1}^{w_n} \ell \, \CExp{\frac{1}{f_n^\ast(k)} \ind{f_n(k)>0}
		\Expn{\abs{\sum_{i,j=1}^n X_{ij}(k,\ell)}^2}^{1/2}}{\Omega_n} \\
	&\le C_2 n^{\delta} \sum_{\ell=1}^{w_n} \ell \, \CExp{\frac{1}{\sqrt{L_n}}}{\Omega_n}\le \frac{C_2 n^{\delta - \frac{1}{2}} w_n^2}{\sqrt{\nu_1 - n^{-\varepsilon}}}.
\end{align*}
Hence, there exists a constant $C_3 > 0$, independent of $k$ such that
\begin{equation}\label{eq:cm_convergence_annd_bound_ell_le_wn}
	\Prob{\Xi_n^{(1)} > \frac{n^{-\delta}}{4}, \Omega_n}
	\le C_3 n^{\delta + 2p - \frac{1}{2}}
	 = C_3 n^{\delta - \frac{\eta}{2(\eta + 2)}}.
\end{equation}

Combining \eqref{eq:cm_convergence_annd_bound_ell_le_wn} and
\eqref{eq:cm_convergence_annd_bound_ell_ge_wn} we get
such that
\begin{align*}
	\Prob{\abs{\Phi_n(k) -\frac{\nu_2}{\nu_1}} > n^{-\delta} ,\Omega_n}
	&\le (C_1 + C_3) n^{\delta -\frac{\eta}{2(\eta + 2)}},
\end{align*}
which proves \eqref{eq:convergence_annd_conditioned_omega}.
\end{proof}

We now proceed with the proof of Lemma~\ref{lem:variance_Xij_kl} which is a stronger version of Lemma 6.3 in~\cite{Hoorn2016b}.
\begin{lemma}\label{lem:variance_Xij_kl}
Let $X_{ij}(k,\ell)$ be as defined in \eqref{eq:convergence_cm_Xij_def}. Then there exist a constant $C > 0$ such that, for any $k, \ell > 0$,
\begin{equation}
\label{eq:variance_Xij_kl}
	\Expn{\left(\sum_{i,j = 1}^n X_{ij}(k,\ell)\right)^2} \le \frac{C f_n^\ast(k)^2 f_n^\ast(\ell)^2}{L_n}.
\end{equation}
\end{lemma}

Using that $f_n^\ast(\ell) \le 1$, we immediately get the next corollary, which we used in the proof of Theorem \ref{thm:convergence_annd_cm_strong}.
\begin{corollary}\label{cor:variance_Xij_kl}
Let $X_{ij}(k,\ell)$ be as defined in \eqref{eq:convergence_cm_Xij_def}. Then there exists constant $C>0$, such that for any $k, \ell > 0$,
\begin{equation*}
\Expn{\left(\sum_{i,j=1}^n X_{ij}(k,\ell)\right)^2} \leq C\,\frac{f_n^\ast(k)^2}{L_n}\,.
\end{equation*}
\end{corollary}

\begin{proof}[Proof of Lemma \ref{lem:variance_Xij_kl}]
Note we can assume that $f_n(k) >0$, otherwise by definition both sides of \eqref{eq:variance_Xij_kl} equal zero. To proceed, define
\begin{equation*}
  Y_{ijst}=\Expn{G_{ij}G_{st}}-\Expn{G_{ij}}\frac{D_sD_t}{L_n}
  -\Expn{G_{st}}\frac{D_iD_j}{L_n}+\frac{D_iD_jD_sD_t}{L_n^2}\,,
\end{equation*}
so that
\begin{equation*}
  \mathbb{E}_n((\sum_{i,j=1}^n X_{ij}(k,\ell))^2)= \frac{1}{L_n^2}\sum_{i,j=1}^n\sum_{s,t=1}^n
  1_{\{D_i=k\}}1_{\{D_j=\ell\}}1_{\{D_s=k\}}1_{\{D_t=\ell\}}Y_{ijst}\,.
\end{equation*}
We will prove the result by considering the different cases for the indices $i, j, s$ and $t$. In the rest of the proof we will write $i \ne j \ne s$ to denote that all 
indices in such inequality are pairwise distinct. Similarly, we write $i = j \ne s \ne t$ to denote all indices $(i,j,s,t)$, where $i = j$ and the three indices $i, s, t$ are pairwise distinct. With this notation we define the sets
\begin{align*}
	I_1 &= \{i,j,s,t : i \ne j \ne s \ne t \}\\
    I_2 &= \{i,j,s,t : i = j \ne s \ne t\} \\
    I_3 &= \{i,j,s,t : i = j = s \ne t\} \\
    I_4 &= \{i,j,s,t : i = j \ne s = t\} \\
    I_5 &= \{i,j,s,t : i = s \ne j \ne t\} \\
    I_6 &= \{i,j,s,t : i = s \ne j = t\}\\
  	I_7 &= \{i,j,s,t : i = j= s = t\}
\end{align*}
Due to symmetry, we can omit other possible cases (e.g. the set $i = j \ne s \ne t$ is equivalent to $i \ne j \ne s = t$). Therefore we only need to consider the sets given above.
Moreover, we can assume without loss of generality that $L_n \ge 4$ so that
\begin{equation}\label{eq:covariance_Xij_bounds_Ln}
	L_n - 1 \ge \frac{L_n}{2}, \quad L_n - 3 \ge \frac{L_n}{4} \quad \text{and} \quad
    2L_n + 3 \le 3L_n.
\end{equation}

\noindent {\bf Case 1:} $\bm{i \ne j \ne s \ne t}$\\
Since all indices are distinct have that
\[
	\Expn{G_{ij}G_{st}} = \frac{D_i D_j D_s D_t}{(L_n - 1)(L_n - 3)},
\]
and hence
\[
	Y_{ijst} = \frac{(2L_n + 3) D_i D_j D_s D_t}{L_n^2 (L_n - 1)(L_n - 3)}.
\]
Using the above and \eqref{eq:covariance_Xij_bounds_Ln} we obtain
\begin{align*}
  	&\hspace{-30pt}\frac{1}{L_n^2}\sum_{i,j,s,t \in I_1}
    	\ind{D_i=k,D_j=\ell,D_s=k,D_t=\ell} Y_{ijst}\\
  	&\le \frac{1}{L_n^2}\sum_{i,j,s,t=1}^n \ind{D_i=k,D_j=\ell,D_s=k,D_t=\ell}
  		\frac{(2L_n + 3) D_i D_j D_s D_t}{L_n^2 (L_n - 1)(L_n - 3)}\\
  	&\le 24\, \sum_{i,j,s,t=1}^n  \ind{D_i=k,D_j=\ell,D_s=k,D_t=\ell}
  		\frac{k^2 \ell^2}{L_n^5} \\
    &= 24\, \frac{1}{L_n^5} \left(\sum_{i = 1}^n \ind{D_i = k} k^2\right)
    	\left(\sum_{j = 1}^n \ind{D_j = \ell} \ell^2\right) \\
    &\le 24\, \frac{1}{L_n^5} \left(\sum_{i = 1}^n \ind{D_i = k} k\right)^2
    	\left(\sum_{j = 1}^n \ind{D_j = \ell} \ell\right)^2 \\
	&= 24\, \frac{f_n^\ast(k)^2 f_n^\ast(\ell)^2}{L_n}.
\end{align*}

\noindent {\bf Case 2:} $\bm{i = j \ne s \ne t}$\\
In this case we have
\[
	\Expn{G_{ij}G_{st}} = \frac{D_i(D_i - 1)D_s D_t}{(L_n - 1)(L_n - 3)} \quad \text{and} \quad
    Y_{ijst} = \frac{3D_i(D_i - 1)D_s D_t}{L_n(L_n - 1)(L_n - 3)}
    - \frac{D_i D_j D_s D_t}{L_n^2 (L_n - 1)}.
\]
In addition, since $i = j$ it follows that the sum over $I_2$ is non-zero if and only if $k = \ell$
and hence
\begin{align*}
	&\hspace{-30pt} \frac{1}{L_n^2} \sum_{i,j,s,t \in I_2} \ind{D_i=k,D_j=\ell,D_s=k,D_t=\ell}
    	Y_{ijst} \\
    &\le \sum_{i,s,t = 1}^n  \ind{D_i=k,D_s=k,D_t=\ell}
    	\frac{3D_i(D_i - 1)D_s D_t}{L_n^3(L_n - 1)(L_n - 3)}\ind{k=\ell} \\
    &\le \sum_{i,s,t = 1}^n  \ind{D_i = D_s = D_t = k = \ell}
    	\frac{24 D_i^2 D_s D_t}{L_n^5} \\
    &= \sum_{i,s,t = 1}^n  \ind{ D_i = D_s = D_t = k=\ell}
    	\frac{24 k^2 \ell^2}{L_n^5} \\
    &\le \frac{24 f_n^\ast(k)^2 f_n^\ast(\ell)^2}{L_n}\,\ind{k = \ell} .
\end{align*}
The computations for the other cases follows in a similar way,
from which the result follows.
\end{proof}

\begin{remark}
Note that we have only proved a pointwise convergence for every $k$. Unfortunately, we cannot directly extend the above method to derive a scaling for 
$\Prob{\sup_{f^\ast(k) > 0} \abs{\Phi_n(k) -\frac{\mu_2}{\mu_1}}> n^{-\delta}}$ because the obtained upper bounds, which are based on 
expectations, are not sharp enough to imply such uniform convergence.
\end{remark}

We proceed with the proof of the central limit theorem for the ANND in the configuration model with regularly-varying degree distribution. In what follows we
use that if $\mathscr{D}$ is regularly varying with exponent $1 < \gamma$, then $\Exp{\mathscr{D}^{1 + \eta}} < \infty$, for $\eta = (\gamma - 1)/2$. In particular
we have the a degree sequence ${\bf D}_n$ generated by $\texttt{IID}(\mathscr{D})$ satisfies Assumption \ref{asmp:convergence_distribution_first_moment}
with $\varepsilon \le (\gamma - 1)/4(\gamma + 3)$.

\begin{proof}[Proof of Theorem \ref{thm:annd_clt_configuration_model}]
By the stable-law CLT, there exists a slowly-varying function $l_0(n)$ such that
\begin{equation}\label{eq:annd_clt_stable_law}
	\frac{\sum_{i = 1}^n D_i^2}{l_0(n) n^{\frac{2}{\gamma}}}
    \dlim S_{\gamma/2},\qquad\mbox{as $n\to\infty$},
\end{equation}
where $S_{\gamma/2}$ is a ${\gamma/2}$-stable random variable.

Let $\Omega_n$ be as in Assumption \ref{asmp:convergence_distribution_first_moment},
with $\varepsilon = (\gamma - 1)/4(\gamma + 3)$. Define the events
\begin{align*}
	A_{nk} = \left\{f_n(k) > 0\right\}, \quad
    B_n = \left\{\sum_{i = 1}^n D_i^2 \le n^{\frac{2}{\gamma} + \frac{\varepsilon}{2}}\right\}.
\end{align*}
and let $\Lambda_{nk} = A_{nk} \cap B_n \cap \Omega_n$. Then $\lim_{n \to \infty} \sup_{k \le n^{\tau}} \Prob{A^c_{nk}} = 0$ by Theorem~\ref{thm:degree_sequence}, and it follows from \eqref{eq:annd_clt_stable_law}, c.f. \cite[Proposition 2.5]{Hoorn2016b}, that $\Prob{B_n} \to 1$,
as $n \to \infty$. Together with Theorem \ref{thm:convergence_distribution_first_moment} this implies that
\begin{equation}\label{eq:annd_clt_convergence_main_event}
	\lim_{n \to \infty} \sup_{k \le n^{\tau}} \Prob{\Lambda^c_{nk}} = 0.
\end{equation}
We now split the main term into three terms as follows:
\begin{align*}
	&\left|\Exp{g\left(\frac{\nu_1 \Phi_n(k)}{l_0(n)
    	n^{\frac{2}{\gamma}-1}}\right) - g\left(S_{2/\gamma}\right)}\right|
    \le \left|\Exp{g\left(\frac{\sum_{i = 1}^n D_i^2}{l(n)n^{\frac{2}{\gamma}}}
    	\right) - g\left(S_{2/\gamma}\right)}\right| \\
    &\hspace{10pt}+ \left|\Exp{g\left(\frac{\nu_1 \sum_{\ell > 0}f_n^\ast(\ell)\ell}{l_0(n)
    	n^{\frac{2}{\gamma}-1}}\right) - g\left(\frac{\sum_{i = 1}^n D_i^2}
    	{l_0(n)n^{\frac{2}{\gamma}}}\right)}\right|\\
    &\hspace{10pt}+ \left|\Exp{g\left(\frac{\nu_1 \Phi_n(k)}{l_0(n)
    	n^{\frac{2}{\gamma}-1}}\right)
        - g\left(\frac{\nu_1 \sum_{\ell > 0}f_n^\ast(\ell)\ell}{l_0(n)
    	n^{\frac{2}{\gamma}-1}}\right)}\right|:= \Xi_n^{(1)} + \Xi_n^{(2)} + \Xi_n^{(3)},
\end{align*}
and we will show that all three terms converge to zero. We remark that by Potter's bounds we have $\lim_{n \to \infty} l_0(n) n^{-\delta} = 0$,
for any $\delta > 0$.

It follows from \eqref{eq:annd_clt_stable_law} that
\[
	\lim_{n \to \infty} \sup_{k \le n^{\tau}} \Xi_n^{(1)} = 0.
\]

We proceed with $\Xi_n^{(2)}$. Note that on the event $\Omega_n$ we have $|L_n - \nu_1 n| \le n^{1 - \varepsilon}$. In addition,
\[
	\sum_{\ell > 0} f_n^\ast(\ell) \ell = \frac{1}{L_n}\sum_{i = 1}^n D_i^2.
\]
Next, since $g$ is bounded and Lipschitz continuous, there exists $C_0 > 0$ such that
\begin{align*}
	\Xi_n^{(2)} &\le \frac{C_0}{l_0(n) n^{\frac{2}{\gamma}}}
    	\Exp{\sum_{i = 1}^n D_i^2\left|\frac{\nu_1 n}{L_n} - 1\right|\ind{\Lambda_{nk}}}
        + 2C_0 \Prob{\Lambda_{nk}^c} \\
    &\le \frac{C_0 n^{\frac{\varepsilon}{2}}}{l_0(n)}
    	\Exp{\frac{|\nu_1 n - L_n|}{L_n}\ind{\Lambda_{nk}}}
        + 2C_0 \Prob{\Lambda_{nk}^c} \le \frac{C_1 n^{\frac{-\varepsilon}{2}}}{l_0(n)}
        + 2C_0 \Prob{\Lambda_{nk}^c}
\end{align*}
for some constant $C_1>0$. The second tern in the last expression converges to zero due to \eqref{eq:annd_clt_convergence_main_event}. It follows that
\[
	\lim_{n \to \infty} \sup_{k \le n^{\tau}} \Xi_n^{(2)} = 0.
\]

Finally, we turn to $\Xi_n^{(3)}$. Using again that $g$ is bounded and Lipschitz continuous, we get
\begin{align*}
	\Xi_n^{(3)}
    &\le \frac{\nu_1 C_0}{l_0(n) n^{\frac{2}{\gamma}-1}} \Exp{\left|\Phi_n(k)
    	- \sum_{\ell > 0}f_n^\ast(\ell)\ell\right|\ind{\Lambda_{n}}}
        + 2C_0 \Prob{\Lambda_{n}^c}.
\end{align*}
Again, the last term converges to zero due to \eqref{eq:annd_clt_convergence_main_event}. For the other term,
similarly to the proof of Theorem \ref{thm:convergence_annd_cm_strong}, we define
\[
	X_{ij}(k,\ell) = \ind{D_i = k, D_j = \ell}\left(\frac{G_{ij}}{L_n}
    - \frac{D_i D_j}{L_n^2}\right)
\]
and use that on $A_{nk}$ we have $f^\ast_n(k) > 0$, to write
\begin{align*}
	\abs{\Phi_n(k) - \sum_{\ell > 0} f_n^\ast(\ell) \ell}
    &= \frac{\abs{\sum_{\ell > 0} h_n(k,\ell)\ell - f_n^\ast(k) f_n^\ast(\ell)\ell}}{f^\ast_n(k)}\le \frac{\sum_{\ell > 0} \ell \abs{\sum_{i,j = 1}^n X_{ij}(k,\ell)}}{f_n^\ast(k)}.
\end{align*}
Taking the expectation, conditioned on the degree sequence, and using Lemma \ref{lem:variance_Xij_kl} we then obtain
\begin{align*}
	\Expn{\abs{\Phi_n(k) - \sum_{\ell > 0} f_n^\ast(\ell) \ell}}
    &\le \frac{\sum_{\ell > 0} \ell }{f_n^\ast(k)}
    	\Expn{ \abs{\sum_{i,j = 1}^n X_{ij}(k,\ell)}} \\
    &\le \frac{\sum_{\ell > 0} \ell}{f_n^\ast(k)}
    	\Expn{\left(\sum_{i,j = 1}^n X_{ij}(k,\ell)\right)^2}^{1/2}\\
    &\le \frac{C}{L_n^{1/2}} \sum_{\ell > 0} f_n^\ast(\ell) \ell = \frac{C}{L_n^{3/2}} \sum_{i = 1}^n D_i^2,
\end{align*}
for some constant $C > 0$. Using again that on the event $\Omega_n$ we have $|L_n - \nu_1 n| \le n^{1 - \varepsilon}$, we obtain
\begin{align*}
	\frac{\nu_1 C_0}{l_0(n) n^{\frac{2}{\gamma}-1}} \Exp{\left|\Phi_n(k)
		- \sum_{\ell > 0}f_n^\ast(\ell)\ell\right|\ind{\Lambda_{n}}}
	&\le \frac{\nu_1 C_0}{l_0(n) n^{\frac{2}{\gamma}-1}} \Exp{\frac{C}{L_n^{3/2}} \sum_{i = 1}^n D_i^2 \ind{\Lambda_{kn}}}\\
	&\le \frac{\nu_1 C_0 C n^{\frac{2}{\gamma}+\frac{\varepsilon}{2}-\frac{3}{2}-\frac{2}{\gamma}+1}}{\left(\nu_1 - n^{- \varepsilon}\right)^{3/2}l_0(n)}
	=\frac{\nu_1 C_0 C n^{-\frac{1 - \varepsilon}{2}}}{\left(\nu_1 - n^{- \varepsilon}\right)^{3/2} l_0(n)}.
\end{align*}
Since we chose $\varepsilon = (\gamma - 1)/4(\gamma + 3) < 1$, it follows that
\[
	\lim_{n \to \infty} \sup_{k \le n^{\tau}} \Xi_n^{(3)} = 0,
\]

which proves the last statement.
\end{proof}

\subsection{Erased configuration model}
\label{sec:proof_ecm}

Denote by $Y_i$ the number of erased stubs of vertex $i$ and let $E_n=\sum_{i=1}^nY_i=L_n-\widehat{L}_n$ be the total number of erased  stubs, which is twice the number of erased undirected edges. We start with two technical results on the relation between the empirical densities $f_n(k)$ and $\widehat{f}_n(k)$.

\begin{lemma}\label{lem:density_ecm_cm_error}
Let $\mathscr{D}$ be regularly varying with exponent $1 < \gamma < 2$ and $\{G_n\}_{n \ge 1}$
be generated by \emph{\texttt{ECM}} with ${\bf D}_n = \emph{\texttt{IID}}(\mathscr{D})$. Then, for any $K, \delta > 0$
\[
	\lim_{n \to \infty} \Prob{\sum_{k = 0}^\infty \left|\widehat{f}_n(k) - f_n(k)\right| > K n^{1 - \gamma + \delta}} = 0.
\]
When $\gamma > 2$, we have
\[
	\lim_{n \to \infty} \Prob{\sum_{k = 0}^\infty \left|\widehat{f}_n(k) - f_n(k)\right| > K n^{-1 + \delta}} = 0.
\]
\end{lemma}

\begin{proof} We write
\begin{align*}
	\sum_{k = 0}^\infty \left|\widehat{f}_n(k) - f_n(k)\right| &\le \sum_{k = 0}^\infty \frac{1}{n} \sum_{i = 1}^n
		\left|\ind{\widehat{D}_i=k} - \ind{D_i = k}\right| \\
	&\le \sum_{k = 0}^\infty \frac{1}{n} \sum_{i = 1}^n \ind{Y_i>0}\left(\ind{\widehat{D}_i = k} + \ind{D_i = k}\right)\\
	&=\frac{1}{n} \sum_{i = 1}^n \ind{Y_i>0} \sum_{k = 0}^\infty \left(\ind{\widehat{D}_i = k} + \ind{D_i = k}\right) = \frac{2}{n} \sum_{i = 1}^n \ind{Y_i>0} \le \frac{2 E_n}{n}.
\end{align*}
Therefore, using \cite[Theorem 8.13]{Hoorn2016b}, we get for $1 < \gamma < 2$,
\begin{align*}
	\lim_{n \to \infty} \Prob{\sum_{k = 0}^\infty \left|\widehat{f}_n(k) - f_n(k)\right| > K n^{1 - \gamma + \delta}}
	&\le \lim_{n \to \infty} \Prob{ {E_n} > \frac{K n^{2 - \gamma + \delta}}{2}} = 0,
\end{align*}
while for $\gamma > 2$ we have
\begin{align*}
	\lim_{n \to \infty} \Prob{\sum_{k = 0}^\infty \left|\widehat{f}_n(k) - f_n(k)\right| > K n^{-1 + \delta}}
	&\le \lim_{n \to \infty} \Prob{ { E_n} > \frac{K n^{\delta}}{2}} = 0,
\end{align*}
\end{proof}

In the above proof we bounded the number of nodes which had a stub removed, $\sum_{i = 1}^n \ind{Y_i > 0}$, with  the
total number of erased stubs $E_n$. Although this is not a tight bound, it is sufficient for our results and we use this
bound several times in proofs regarding the ECM. The next lemma controls the event on which the empirical degree and
size-biased degree distribution are non-zero, given that the limit distribution is non-zero.

\begin{lemma}\label{lem:non_zero_empirical_densities}
Let $\mathscr{D}$ be regularly varying with exponent $\gamma > 1$ and $\{G_n\}_{n \ge 1}$
be generated by \emph{\texttt{ECM}} with ${\bf D}_n = \emph{\texttt{IID}}(\mathscr{D})$. Then, if we denote by
$f(k)$ the probability density function of $\mathscr{D}$, and let $k$ be such that $f(k) > 0$,
\[
	\lim_{n \to \infty} \Prob{\widehat{f}_n(k) > \frac{f(k)}{2}, f_n(k) > \frac{f(k)}{2}} = 1,
\]
and in particular
\[
	\lim_{n \to \infty} \Prob{\widehat{f}_n(k) > 0, f_n(k) > 0} = 1.
\]
\end{lemma}

\begin{proof}
For the first result we show that
\[
	\lim_{n \to \infty} 1 - \Prob{\widehat{f}_n(k) > \frac{f(k)}{2}, f_n(k) > \frac{f(k)}{2}} = 0,
\]
which also implies the second result since $f(k) > 0$. By splitting the probability we get
\begin{align*}
	&\hspace{-30pt}1 - \Prob{\widehat{f}_n(k) > \frac{f(k)}{2}, f_n(k) > \frac{f(k)}{2}, \Lambda_n} \\
	&\le \Prob{\widehat{f}_n(k) \le \frac{f(k)}{2}}
		+ \Prob{f_n(k) \le \frac{f(k)}{2}}\\
	&\le \Prob{|\widehat{f}_n(k) - f(k)| \ge \frac{f(k)}{2}} + \Prob{\left|f_n(k) - f(k)\right| \ge \frac{f(k)}{2}} \\
	&\le \Prob{|\widehat{f}_n(k) - f_n(k)| \ge \frac{f(k)}{4}} + 2\Prob{\left|f_n(k) - f(k)\right| \ge \frac{f(k)}{4}}.
\end{align*}
Since the first probability converges to zero by Lemma \ref{lem:density_ecm_cm_error} while this holds the second term
by Theorem \ref{thm:convergence_distribution_first_moment}, the result follows.
\end{proof}

\begin{proof}[Proof of Theorem~\ref{thm:annd_erased_model_error_term}]
Recall that $a = (\gamma - 1)^2/(2\gamma)$ and note that $f^\ast(k) > 0$ implies that $f(k) > 0$ and $k \ge 1$. Let $\Omega_n$ be as in Assumption \ref{asmp:convergence_distribution_first_moment} with $\varepsilon = (\gamma - 1)/4(\gamma + 3)$ and define the events
\begin{align*}
	A_n &= \left\{\max_{1 \le i \le n} D_i \le n^{\frac{1}{\gamma} + \frac{a}{4}}\right\}, \\
	B_n &= \left\{\Expn{{ E_n}} \le n^{2 - \gamma + \frac{a}{4}}\right\}, \\
	C_n &= \left\{\widehat{f}_n(k) > \frac{f(k)}{2}, f_n(k) > \frac{f(k)}{2} \right\}.
\end{align*}
Then $\Prob{A_n^c} \to 0$, since $D_i$ are i.i.d. samples from $\mathscr{D}$, while
$\Prob{B_n^c} \to 0$ by Theorem 8.13 in~\cite{Hoorn2016b}. Finally, Lemma \ref{lem:non_zero_empirical_densities}
implies that $\Prob{C_n} \to 1$ so that
if we define $\Lambda_n = A_n \cap B_n \cap C_n \cap \Omega_n$, then
\[
	\Prob{\Lambda^c_n} \le \Prob{\Omega_n^c} + \Prob{A_n^c} + \Prob{B_n^c} + \Prob{C^c_n}
	\to 0,
\]
as $n \to \infty$. Hence it is sufficient to prove the result conditioned on the event $\Lambda_n$.

By definition, and since $\widehat{f}_n^\ast(k), f^\ast(k) > 0$ on the event $\Lambda_n$, we have that
\[
	\widehat{\Phi}_n(k) = \frac{\sum_{i,j = 1}^n \widehat{G}_{ij} \widehat{D}_j
	\ind{\widehat{D}_i = k}}{k \sum_{i = 1}^n \ind{\widehat{D}_i = k}},
\]
and we split $|\widehat{\Phi}_n(k) - \Phi_n(k)|$ in four terms as
follows
\begin{align}
	\left|\widehat{\Phi}_n(k) - \Phi_n(k)\right|
	&\le \left|\frac{1}{k \sum_{i = 1}^n \ind{\widehat{D}_i = k}}
		- \frac{1}{k \sum_{i = 1}^n \ind{D_i = k}}\right|
		\sum_{i,j = 1}^n \widehat{G}_{ij} \widehat{D}_j \ind{\widehat{D}_i = k}
		\label{eq:phi_bound_numerator}\\
	&\hspace{10pt}+ \frac{1}{k \sum_{i = 1}^n \ind{D_i = k}}\sum_{i,j = 1}^n \widehat{G}_{ij}
		\widehat{D}_j\left|\ind{\widehat{D}_i = k} - \ind{D_i = k}\right|
		\label{eq:phi_bound_indicator}\\
	&\hspace{10pt}+  \frac{1}{k \sum_{i = 1}^n \ind{D_i = k}}\sum_{i,j = 1}^n \widehat{G}_{ij}
		\ind{D_i = k}\left|\widehat{D}_j - D_j\right| \label{eq:phi_bound_degree}\\
	&\hspace{10pt}+  \frac{1}{k \sum_{i = 1}^n \ind{D_i = k}}\sum_{i,j = 1}^n D_j
	\ind{D_i = k}\left| \widehat{G}_{ij} - G_{ij}\right|. \label{eq:phi_bound_edges}
\end{align}
The remainder of the proof consists of bounding each of the terms \eqref{eq:phi_bound_numerator}--\eqref{eq:phi_bound_edges} by an expression of the form $c n^{\frac{1}{\gamma} + \frac{a}{4} - 1} E_n$. This will give the result because for any constant $c > 0$ we have

\begin{align*}
	\lim_{n \to \infty} \Prob{ c n^{\frac{1}{\gamma} + \frac{a}{4} - 1} E_n
		> n^{\frac{2}{\gamma} -1 - a}, \Lambda_n}
	&\le \lim_{n \to \infty} c n^{\frac{5a}{4} - \frac{1}{\gamma}} \CExp{\Expn{{ E_n}}}{\Lambda_n}\\
	&\le \lim_{n \to \infty} c n^{\frac{3a}{2} - \gamma + 2 - \frac{1}{\gamma}} = 0, \numberthis
		\label{eq:bounding_phi_scaling_erased_edges}
\end{align*}
where the convergence to zero holds because
\[
\frac{3a}{2} - \gamma + 2 - \frac{1}{\gamma} = \frac{3a}{2} - \frac{(\gamma - 1)^2}{\gamma}
= \frac{3a}{2} - 2a = -\frac{a}{2} < 0.
\]
\smallskip

In order to bound \eqref{eq:phi_bound_numerator}, we first use that $|\widehat{f}_n(k) - f_n(k)| \le E_n/n$ to obtain
\[
\left|\frac{1}{k \sum_{i = 1}^n \ind{\widehat{D}_i = k}} - \frac{1}{k\sum_{i = 1}^n \ind{D_i = k}}\right|
= \frac{\left| f_n(k) - \widehat{f}_n(k) \right|}{n k \widehat{f}_n(k) f_n(k)}
\le \frac{2E_n}{n^2 \widehat{f}_n^\ast(k) f_n^\ast(k)}.
\]
Then, on $\Lambda_n$ and using that $\widehat{D}_i \le D_i$, we get
\begin{align*}
	&\hspace{-30pt}\left|\frac{1}{k \sum_{i = 1}^n \ind{\widehat{D}_i = k}} - \frac{1}{k\sum_{i = 1}^n \ind{D_i = k}}\right|
		\sum_{i,j = 1}^n \widehat{G}_{ij} \widehat{D}_j \ind{\widehat{D}_i = k}\\
	&\le \frac{2E_n n^{\frac{1}{\gamma} + \frac{a}{4}}}{n^2 k \widehat{f}_n(k) f_n(k)} \sum_{i,j = 1}^n \widehat{G}_{ij}
		\ind{\widehat{D}_i = k}\\
	&= \frac{2E_n n^{\frac{1}{\gamma} + \frac{a}{4}}}{n^2 k \widehat{f}_n(k) f_n(k)} \sum_{i = 1}^n \widehat{D}_i
		\ind{\widehat{D}_i = k} \\
	&= \frac{2E_n n^{\frac{1}{\gamma} + \frac{a}{4} - 1}}{f_n(k)} \le \frac{4 }{f(k)}\, n^{\frac{1}{\gamma} + \frac{a}{4} - 1} E_n.
\end{align*}
For \eqref{eq:phi_bound_indicator}, on the event $\Lambda_n$, using that $\widehat{G}_{ij} \le G_{ij} \wedge 1$, we get
\begin{align*}
	\frac{1}{k \sum_{i = 1}^n \ind{D_i = k}}&\sum_{i, j = 1}^n \widehat{G}_{ij} D_j\left|\ind{\widehat{D}_i = k} -
		\ind{D_i = k}\right| \\
	&\le \frac{1}{n k f_n(k)}\,n^{\frac{1}{\gamma} + \frac{a}{4}} \sum_{i, j = 1}^n \widehat{G}_{ij} \left(
		\ind{\widehat{D}_i = k} + \ind{D_i = k}\right) \ind{Y_i > 0} \\
	&\le \frac{1}{n k f_n(k)}\,2 k n^{\frac{1}{\gamma} + \frac{a}{4}} \sum_{i = 1}^n \ind{Y_i > 0}\le \frac{4}{f(k)} n^{\frac{1}{\gamma} + \frac{a}{4}-1} E_n.
\end{align*}
In order to bound \eqref{eq:phi_bound_degree}, note that $|\widehat{D}_i - D_i| = Y_i$. Hence, since $k \ge 1$, we have
\begin{align*}
	\frac{1}{k \sum_{i = 1}^n \ind{D_i = k}}\sum_{i,j = 1}^n \widehat{G}_{ij}
		\ind{D_i = k}\left|\widehat{D}_j - D_j\right|
	&= \frac{1}{n k f_n(k)} \sum_{i,j = 1}^n \widehat{G}_{ij} Y_j \\
	&\le \frac{1}{n k f_n(k)} \sum_{j = 1}^n D_j Y_j \le \frac{2}{f(k)} n^{\frac{1}{\gamma} + \frac{a}{4} - 1} E_n,
\end{align*}

Finally, for \eqref{eq:phi_bound_edges}, on the event $\Lambda_n$, we obtain
\begin{align*}
	\frac{1}{k \sum_{i = 1}^n \ind{D_i = k}}\sum_{i,j = 1}^n D_j
		\ind{D_i = k}\left| \widehat{G}_{ij} - G_{ij}\right|
	&\le \frac{1}{n k f_n(k)} n^{\frac{1}{\gamma} + \frac{a}{4}} \sum_{i = 1}^n Y_i
	\le \frac{2}{f(k)} n^{\frac{1}{\gamma} + \frac{a}{4} - 1} E_n.
\end{align*}
It follows that the sum of the terms \eqref{eq:phi_bound_numerator}--\eqref{eq:phi_bound_edges} is bounded from above by $c n^{\frac{1}{\gamma} + \frac{a}{4} - 1} E_n$, where
\[c=\frac{4}{f(k)}+\frac{4}{f(k)}+\frac{2}{f(k)}+\frac{2}{f(k)}=\frac{12}{f(k)}.\]
The result now follows from \eqref{eq:bounding_phi_scaling_erased_edges}.
\end{proof}

\subsection{Average nearest neighbor rank}\label{ssec:proofs_annr}

\begin{proof}[Proof of Theorem \ref{thm:annr_convergence_general}]
By Assumption \ref{asmp:regularity_structure} there exists $\kappa > 0$ such
that if
\[
	\Gamma_n = \left\{\sum_{k,\ell > 0} \abs{h_n(k,\ell) - h(k,\ell)}
    \ind{f_n(k) > 0} \le n^{-\kappa}\right\},
\]
then $\Prob{\Gamma_n} \to 1$. Now let $\Omega_n$ denote the event from Assumption \ref{asmp:convergence_distribution_first_moment}.
%, with $\varepsilon = \eta/(8 + 4\eta)$. 
Then, if we define $\Lambda_n = \Gamma_n \cap \Omega_n$, it follows 
%from Theorem \ref{thm:convergence_distribution_first_moment} 
that $\Prob{\Lambda_n} \to 1$. Therefore, we only need to prove the result, conditioned on $\Lambda_n$.

By definition of $\Theta_n(k)$ and $\Theta(k)$ we have
\begin{align*}
  \abs{\Theta_n(k)-\Theta(k)}&1_{\{f_n(k)>0\}}
  \leq \sum_{\ell > 0} h_n(k,\ell)F_n^\ast(\ell)
  	\abs{\frac{1}{f_n^\ast(k)}-\frac{1}{f^\ast(k)}}\ind{f_n(k)>0} \\
  &\hspace{10pt}+ \frac{1}{f^\ast(k)}\abs{\sum_{\ell > 0} h_n(k,\ell)F_n^\ast(\ell)
  	- h(k,\ell)F^\ast(\ell)}\ind{f_n(k)>0} \\
  &\le \frac{\ind{f_n(k)>0}}{f_n^\ast(k) f^\ast(k)}\sum_{\ell > 0} h_n(k,\ell)F_n^\ast(\ell)
  	\abs{f_n^*(k)-f^\ast(k)} \\
  &+  \frac{1}{f^\ast(k)} \sum_{\ell > 0}
  	\abs{h_n(k,\ell)-h(k,\ell)}F_n^\ast(\ell)\ind{f_n(k)>0} + \frac{1}{f^\ast(k)}\sum_{\ell > 0} h(k,\ell)
  \abs{F_n^\ast(\ell)-F^\ast(\ell)}  \\
  &:= \Xi_n^{(1)} + \Xi_n^{(2)} + \Xi_n^{(3)}.
\end{align*}
We will show that
\[
	\lim_{n \to \infty}	 \Prob{\Xi_n^{(i)} > \frac{n^{-\delta}}{3}, \Lambda_n} = 0 \quad \text{for} \quad i = 1,2,3.
\]
The result then follows because
\begin{align*}
	&\hspace{-30pt}\Prob{\abs{\Theta_n(k)-\Theta(k)}1_{\{f_n(k)>0\}} > n^{-\delta}, \Lambda_n}\\
    &\le \Prob{\Xi_n^{(1)} > \frac{n^{-\delta}}{3}, \Lambda_n}
    + \Prob{\Xi_n^{(2)} > \frac{n^{-\delta}}{3}, \Lambda_n}
    + \Prob{\Xi_n^{(3)} > \frac{n^{-\delta}}{3}, \Lambda_n}.
\end{align*}
We start with $\Xi_n^{(1)}$. On the event $\Omega_n$ we have that $\abs{f_n^\ast(k) - f^\ast(k)} \le n^{-\varepsilon}$ and hence, since
\[
	\sum_{\ell > 0} h_n(k,\ell)F_n^\ast(\ell) \le 1,
\]
we have
\begin{align*}
	\Xi_n^{(1)} &\le \frac{\ind{f_n(k)>0}\abs{f_n^\ast(k) - f^\ast(k)}}
    	{f^\ast(k)(f^\ast(k) - n^{-\varepsilon})} \sum_{\ell > 0} h_n(k,\ell)F_n^\ast(\ell)
    \le \frac{d_{tv}(f_n^\ast, f^\ast)}
    	{f^\ast(k)(f^\ast(k) - n^{-\varepsilon})} \le \frac{n^{-\varepsilon}}{f^\ast(k)(f^\ast(k) - n^{-\varepsilon})}.
\end{align*}
Therefore, since $\delta < \varepsilon$ it follows that
\[
	\lim_{n \to \infty} \Prob{\Xi_n^{(1)} > \frac{n^{-\delta}}{3}, \Lambda_n} = 0.
\]
For $\Xi_n^{(2)}$ we note that
\[
	\Xi_n^{(2)} \le \frac{1}{f^\ast(k)} \sum_{k, \ell > 0} \abs{h_n(k,\ell) - h(k,\ell)}\ind{f_n(k)>0}
\]
and hence, since $\delta < \kappa$,
\[
	\lim_{n \to \infty} \Prob{\Xi_n^{(2)} > \frac{n^{-\delta}}{3}, \Lambda_n} = 0.
\]
Finally, for $\Xi_n^{(3)}$ we use \eqref{eq:sup_bound_total_variantion} to get
\[
	\Xi_n^{(3)} \le \frac{\sup_{\ell > 0} \abs{F_n^\ast(\ell) - F^\ast(\ell)}}{f^\ast(k)}
    \sum_{\ell > 0} h(k,\ell) \le \frac{2 d_{tv}(f_n^\ast, f^\ast)}{f^\ast(k)} \le \frac{2 n^{-\varepsilon}}{f^\ast(k)}
\]
which implies that
\[
	\lim_{n \to \infty} \Prob{\Xi_n^{(3)} > \frac{n^{-\delta}}{3}, \Lambda_n} = 0.
\]
\end{proof}

\begin{proof}[Proof of Theorem \ref{thm:annr_convergence_cm}]
Recall that
\[
	\Exp{F^\ast(\mathscr{D}^\ast)} = \sum_{\ell > 0} f^\ast(\ell) F^\ast(\ell).
\]
Then we have
\begin{align*}
	\abs{\Theta_n(k)-\Exp{F^\ast({\mathscr{D}}^\ast)}}&\ind{f_n(k)>0}
    = \frac{\ind{f_n(k) > 0}}{f_n^\ast(k)}\abs{\sum_{\ell > 0} h_n(k,\ell)F_n^\ast(\ell)
    	- f^\ast_n(k)f^\ast(\ell) F^\ast(\ell)} \\
    &\le \frac{\ind{f_n(k) > 0}}{f_n^\ast(k)} \abs{\sum_{\ell > 0} (h_n(k,\ell)
    	- f_n^\ast(k)f_n^\ast(\ell))F_n^\ast(\ell) } \\
    &\hspace{10pt}+ \abs{\sum_{\ell > 0} f_n^\ast(\ell)F_n^\ast(\ell)-f^\ast(\ell)F^\ast(\ell)} \\
   	&\le \frac{\ind{f_n(k) > 0}}{f_n^\ast(k)} \abs{\sum_{\ell > 0} (h_n(k,\ell)
    	- f_n^\ast(k)f_n^\ast(\ell))F_n^\ast(\ell) } \\
    &\hspace{10pt}+ \sum_{\ell > 0} \abs{f_n^\ast(\ell)-f^\ast(\ell)} + \sum_{\ell > 0} \abs{F_n^\ast(\ell)-F^\ast(\ell)}f^\ast(\ell) \\
   	&:= \Xi_n^{(1)} + \Xi_n^{(2)} + \Xi_n^{(3)}.
\end{align*}
First, note that $\Xi_n^{(2)}$, $\Xi_n^{(3)}$ are independent of $k$ and
\[
	\Xi_n^{(2)} = 2 d_{tv}(f_n^\ast, f^\ast) \quad \text{and} \quad
    \Xi_n^{(3)} \le \sup_{\ell > 0} \abs{F_n^\ast(\ell) - F^\ast(\ell)}
    \le d_{tv}(f_n^\ast, f^\ast).
\]
Now let $\Omega_n$ is be defined as in Assumption \ref{asmp:convergence_distribution_first_moment} with
$\varepsilon = \eta/(8 + 4\eta)$. Then, by Theorem \ref{thm:convergence_distribution_first_moment}
\[
	\Prob{\Omega_n^c} = \bigO{n^{-\varepsilon}} = \bigO{n^{-\delta}},
\]
with all constants independent of $k$. In particular, since $\delta < \eta/(8 + 4\eta)$, this implies that
\[
	\sup_{k > 0} \Prob{\Xi_n^{(2)} + \Xi_n^{(3)}
    > \frac{n^{-\delta}}{2}} \le 2\Prob{d_{tv}(f_n^\ast, f^\ast) > \frac{n^{-\delta}}{4}}
    = \bigO{n^{-\delta}}.
\]

For $\Xi_n^{(1)}$ we will show that
\[
	\sup_{k > 0} \Prob{\Xi_n^{(1)}
    > \frac{n^{-\delta}}{2}, \Omega_n} = O\left(n^{\delta - \frac{1}{2}} \right) = O\left(n^{-\delta}\right),
\]
since $\delta < \eta/(8 + 4\eta) < 1/4$. For this we will use the same approach as for the proof of Theorem \ref{thm:convergence_annd_cm_strong}. That is, we split the summation
into the regimes $\ell \le w_n$ and $\ell > w_n$, with $w_n = \floor{n^p}$, and invoke Lemma \ref{lem:variance_Xij_kl}.
Recall the definition of $X_{ij}(k,\ell)$,
\[
	X_{ij}(k,\ell) = \ind{D_i = k, \, D_j = \ell} \left(\frac{G_{ij}}{L_n}
    -\frac{D_i D_j}{L_n^2}\right).
\]
Then, since $f_n^\ast(\ell) \le 1$,
\begin{align*}
	\Xi_n^{(1)} &\le \frac{\ind{f_n(k) > 0}}{f_n^\ast(k)}
    	\abs{\sum_{\ell > 0} h_n(k,\ell) - f_n^\ast(k)f_n^\ast(\ell)} \\
    &\le \frac{\ind{f_n(k) > 0}}{f_n^\ast(k)} \sum_{\ell \le w_n}
    	\abs{\sum_{i,j = 1}^n X_{ij}(k,\ell)}
    	+ \frac{\ind{f_n(k) > 0}}{f_n^\ast(k)} \sum_{\ell > w_n}
    	\abs{\sum_{i,j = 1}^n X_{ij}(k,\ell)}.
\end{align*}
Taking the conditional expectation and using Lemma \ref{lem:variance_Xij_kl}, in a similar way as in the proof of Theorem \ref{thm:convergence_annd_cm_strong}, we obtain
\begin{align*}
	\Expn{\Xi_n^{(1)}}
    &\le C \sum_{\ell \le w_n} \frac{f_n^\ast(\ell)}{L_n^{1/2}}
    	+ 2\sum_{\ell > w_n} \sum_{i = 1}^n \ind{D_i = \ell}\frac{D_i}{L_n} \\
    &= \frac{C}{L_n^{3/2}} \sum_{i = 1}^n D_i \ind{D_i \le w_n}
    	+\frac{2}{L_n} \sum_{i = 1}^n D_i \ind{D_i > w_n}\\
    &\le \frac{C}{L_n^{1/2}} +\frac{2}{L_n} \sum_{i = 1}^n D_i \ind{D_i > w_n},
\end{align*}
for some constant $C > 0$, which is independent from $k$. On the event $\Omega_n$ we have that
$L_n > \nu_1 n - n^{1 - \varepsilon}$, while
\[
	\Exp{\mathscr{D} \ind{\mathscr{D} > w_n}} \le w_n^{-\eta} \Exp{\mathscr{D}^{1 + \eta}}
    = \bigO{n^{-\eta p}}.
\]
Therefore we get
\begin{align*}
	\Prob{\Xi_n^{(1)} > \frac{n^{-\delta}}{2}, \Omega_n}
    &\le 2n^{\delta} \CExp{\Xi_n^{(1)}}{\Omega_n} \le \frac{2n^{\delta}}{\Prob{\Omega_n}} \Exp{\Expn{\Xi_n^{(1)}}\ind{\Omega_n}} \\
    &\le \frac{2Cn^{\delta}}{(\nu_1 n - n^{-\varepsilon})^{1/2}\Prob{\Omega_n}}
    	+ \frac{4n^{\delta}}{(\nu_1 - n^{-\varepsilon})\Prob{\Omega_n}}
        \Exp{\mathscr{D} \ind{\mathscr{D} > w_n}} \\
    &= \bigO{n^{\delta - \frac{1}{2}} + n^{\delta - \eta p}},
\end{align*}
and hence, by taking $p = 2\delta/\eta$, we obtain that for any $0 < \delta < \eta/(8+4\eta)$,
\begin{align*}
	\sup_{k > 0} \Prob{\abs{\Theta_n(k) - \Exp{F^\ast(\mathscr{D}^\ast)}} > n^{-\delta}}
    = \bigO{n^{-\delta/2} + n^{\delta - \frac{1}{2}} + n^{-\delta}} = \bigO{n^{\delta - \frac{1}{2}}}.
\end{align*}
\end{proof}

\begin{proof}[Proof of Theorem \ref{thm:annr_erased_model_error_term}]
Let $\Omega_n$ be defined as in Assumption \ref{asmp:convergence_distribution_first_moment} for some $\varepsilon > 0$ and define the events
\begin{align*}
	A_n = \left\{\Expn{E_n} \le n^{\frac{3 - \gamma}{2}}\right\}, \quad
	B_n = \left\{\widehat{f}_n(k) > \frac{f(k)}{2}, f_n(k) > \frac{f(k)}{2}\right\}.
\end{align*}
Then, since $(3-\gamma)/2 > 2 - \gamma$, $\Prob{A_n} \to 1$ by \cite[Theorem 8.13]{Hoorn2016b}, while $\Prob{B_n} \to 1$
by Lemma \ref{lem:non_zero_empirical_densities}. Hence, if we define $\Lambda_n = A_n \cap B_n \cap \Omega_n$,
then $\Prob{\Lambda_n} \to 1$,
as $n \to \infty$ so that it is enough to prove the result conditioned on the event $\Lambda_n$.

By definition we have that
\[
	\widehat{\Theta}_n(k) = \frac{\sum_{i,j = 1}^n \widehat{G}_{ij} \widehat{F}^*_n(\widehat{D}_j)
	\ind{\widehat{D}_i = k}}{k\sum_{i = 1}^n \ind{\widehat{D}_i = k}} \ind{\widehat{f}_n(k) > 0},
\]
where the hats denote expressions for the erased model. As in the proof of Theorem
\ref{thm:annd_erased_model_error_term}, we split $|\widehat{\Theta}_n(k) - \Theta_n(k)|$ in four terms
\begin{align}
	\left|\widehat{\Theta}_n(k) - \Theta_n(k)\right|
	&\le \left|\frac{\ind{\widehat{f}_n(k) > 0}}{k\sum_{i = 1}^n \ind{\widehat{D}_i = k}}
		- \frac{\ind{f_n(k) > 0}}{k\sum_{i = 1}^n \ind{D_i = k}}\right|
		\sum_{i,j = 1}^n \widehat{G}_{ij} \widehat{F}^*_n(\widehat{D}_j) \ind{\widehat{D}_i = k}
		\label{eq:theta_bound_numerator}\\
	&\hspace{10pt}+ \frac{\ind{f_n(k) > 0}}{k\sum_{i = 1}^n \ind{D_i = k}}\sum_{i,j = 1}^n \widehat{G}_{ij}
		\widehat{F}^*_n(\widehat{D}_j)\left|\ind{\widehat{D}_i = k} - \ind{D_i = k}\right|
		\label{eq:theta_bound_indicator}\\
	&\hspace{10pt}+  \frac{\ind{f_n(k) > 0}}{k\sum_{i = 1}^n \ind{D_i = k}}\sum_{i,j = 1}^n F^*_n(D_j)
		\ind{D_i = k}\left| \widehat{G}_{ij} - G_{ij}\right| \label{eq:theta_bound_edges}\\
	&\hspace{10pt}+  \frac{\ind{f_n(k) > 0}}{k\sum_{i = 1}^n \ind{D_i = k}}\sum_{i,j = 1}^n \widehat{G}_{ij}
		\ind{D_i = k}\left|\widehat{F}^*_n(\widehat{D}_j) - F^*_n(D_j)\right| \label{eq:theta_bound_degree}
\end{align}
and bound all four terms individually.

First, we bound the terms \eqref{eq:theta_bound_numerator}--\eqref{eq:theta_bound_edges} by the expressions of the form $\frac{C E_n}{f(k)n}$ for some $C>0$ and use that, for any $c > 0$ and $f(k) > 0$,
\begin{align*}
	\lim_{n \to \infty} \Prob{\frac{E_n}{n} > c f(k), \Lambda_n}
	\le \lim_{n \to \infty} \frac{1}{c f(k) n} \CExp{\Expn{E_n}}{A_n }
	= \lim_{n \to \infty} \frac{n^{-\frac{\gamma - 1}{2}}}{c f(k)} = 0. \numberthis
		\label{eq:bounding_theta_scaling_erased_edges}
	\end{align*}

For \eqref{eq:theta_bound_numerator}, on $\Lambda_n$ we have $\widehat{f}_n(k), f_n(k) > 0$ while $\left| f_n(k)-\widehat{f}_n(k) \right|\leq E_n/n$.
Therefore we obtain
\[
\left|\frac{1}{k\sum_{i = 1}^n \ind{\widehat{D}_i = k}} - \frac{1}{k\sum_{i = 1}^n \ind{D_i = k}}
\right| = \frac{\left| f_n(k) - \widehat{f}_n(k) \right|}{n k \widehat{f}_n(k) f_n(k)}
\le \frac{E_n}{n^2 k \widehat{f}_n(k) f_n(k)}.
\]
From this, and $\widehat{F}^\ast_n \le 1$, it follows that
\begin{align*}
	&\hspace{-60pt}\left|\frac{1}{k\sum_{i = 1}^n \ind{\widehat{D}_i = k}} - \frac{1}{k\sum_{i = 1}^n
		\ind{D_i = k}}\right|\sum_{i,j = 1}^n \widehat{G}_{ij} \widehat{F}^\ast_n(\widehat{D}_j)
		\ind{\widehat{D}_i = k} \\
	&\le \frac{E_n}{n^2 k \widehat{f}_n(k) f_n(k)} \sum_{i,j = 1}^n \widehat{G}_{ij}
		\ind{\widehat{D}_i = k} \\
	&\le \frac{E_n}{n^2 k \widehat{f}_n(k) f_n(k)} \sum_{i = 1}^n \widehat{D}_i
		\ind{\widehat{D}_i = k}\\
	&\le \frac{E_n}{nf_n(k)} \le \frac{2E_n}{n f(k)}.
\end{align*}

Continuing with \eqref{eq:theta_bound_indicator}, on the event $\Lambda_n$, using that $\widehat{G}_{ij} \le G_{ij} \wedge 1$, we obtain
\begin{align*}
	\frac{1}{k \sum_{i = 1}^n \ind{D_i = k}}\,&\sum_{i, j = 1}^n \widehat{G}_{ij} F^*_n(D_j)\left|\ind{\widehat{D}_i = k} -
		\ind{D_i = k}\right|\\
	&\le \frac{1}{k n  f_n(k)} \sum_{i, j = 1}^n \widehat{G}_{ij} \left(
		\ind{\widehat{D}_i = k} + \ind{D_i = k}\right) \ind{Y_i > 0} \\
	&\le  \frac{2 k}{k n f_n(k)} \sum_{i = 1}^n \ind{Y_i > 0} \le \frac{4  E_n}{n f(k)}.
\end{align*}

For \eqref{eq:theta_bound_edges}, we get, on the event $\Lambda_n$,
\begin{align*}
	\frac{1}{k \sum_{i = 1}^n \ind{D_i = k}}\sum_{i,j = 1}^n F^*_n(D_j)
		\ind{D_i = k}\left| \widehat{G}_{ij} - G_{ij}\right|
	&\le \frac{1}{n f_n(k)}  \sum_{i = 1}^n Y_i
	= \frac{E_n}{n f_n(k)} \le \frac{2 E_n}{n f(k)}.
\end{align*}

Now, using \eqref{eq:bounding_theta_scaling_erased_edges}, we obtain that
\[
	\lim_{n \to \infty} \Prob{\mbox{\eqref{eq:theta_bound_numerator}}+\mbox{\eqref{eq:theta_bound_indicator}}+\mbox{\eqref{eq:theta_bound_edges}} >
		\frac{\delta}{2},\, \Lambda_n}\le \lim_{n \to \infty} \Prob{\frac{  E_n}{n} > \frac{\delta}{16}\, f(k),\,\Lambda_n}= 0.
\]

It remains to bound \eqref{eq:theta_bound_degree}. This requires a bit more work. First we split $\left|\widehat{F}^\ast_n(\widehat{D}_j) - F^\ast_n(D_j)\right|$ as
\begin{align*}
	\left|\widehat{F}^\ast_n(\widehat{D}_j) - F^\ast_n(D_j)\right| &\le \left|\widehat{F}^\ast_n(\widehat{D}_j) - F^\ast_n(\widehat{D}_j)\right|
	+ \left|F^\ast_n(\widehat{D}_j) - F^\ast_n(D_j)\right|
\end{align*}
For the last term we have
\begin{align*}
	\left|F^\ast_n(\widehat{D}_j) - F^\ast_n(D_j)\right|
	&\le \frac{1}{L_n} \sum_{i = 1}^n D_i \left|\ind{D_i \le \widehat{D}_j} - \ind{D_i \le D_j}\right| \\
	&\le  \frac{2}{L_n} \sum_{i = 1}^n D_i \ind{D_i \le D_j} \ind{Y_j > 0}= 2 F_n^\ast(D_j) \ind{Y_j > 0}.
\end{align*}
For the other term we have
\begin{align*}
	\left|\widehat{F}^\ast_n(\widehat{D}_j) - F^\ast_n(\widehat{D}_j)\right|
	&\le \left|\frac{1}{\widehat{L}_n} - \frac{1}{L_n}\right|\sum_{i = 1}^n \widehat{D}_i \ind{\widehat{D}_i \le \widehat{D}_j}
		+ \frac{1}{L_n} \sum_{i = 1}^n  \left|\widehat{D}_i \ind{\widehat{D}_i \le \widehat{D}_j} - D_i \ind{D_i \le \widehat{D}_j}\right|\\
	&\le \frac{E_n}{L_n} \widehat{F}_n^\ast(\widehat{D}_j)
		+ \frac{1}{L_n} \sum_{i = 1}^n  Y_i \ind{\widehat{D}_i \le \widehat{D}_j} + \frac{1}{L_n} \sum_{i = 1}^n D_i
		\left|\ind{\widehat{D}_i \le \widehat{D}_j}  - \ind{D_i \le \widehat{D}_j} \right|\\
	&\le \frac{2 E_n}{L_n} + \frac{2}{L_n} \sum_{i = 1}^n D_i \ind{Y_i > 0}.
\end{align*}
To summarize, we have
\[
	\left|\widehat{F}^\ast_n(\widehat{D}_j) - F^\ast_n(D_j)\right| \le  \frac{2 E_n}{L_n} + 2 F_n^\ast(D_j) \ind{Y_j > 0}
	+ \frac{2}{L_n} \sum_{i = 1}^n D_i \ind{Y_i > 0}.
\]
Next we note that
\[
	\frac{1}{n k f_n(k)}\sum_{i,j = 1}^n \widehat{G}_{ij} \ind{D_i=k}
	\le \frac{1}{n k f_n(k)}\sum_{i = 1}^n D_i \ind{D_i=k} = 1.
\]
Hence, on the event $\Lambda_n$,
\begin{align}
	\nonumber &\hspace{-30pt}\frac{1}{k\sum_{i = 1}^n \ind{D_i = k}}\sum_{i,j = 1}^n \widehat{G}_{ij}
		\ind{D_i = k}\left|\widehat{F}^*_n(\widehat{D}_j) - F^*_n(D_j)\right|\\
	\nonumber &\le \frac{2E_n}{L_n} + \frac{1}{L_n} \sum_{i = 1}^n D_i \ind{Y_i > 0}
		+ \frac{2}{n k f_n(k)}\sum_{i,j = 1}^n \widehat{G}_{ij} F_n^\ast(D_j) \ind{Y_j > 0} \ind{D_i=k}\\
	\nonumber &\le \frac{2E_n}{(\nu_1 n - n^{1 - \varepsilon})} + \frac{1}{(\nu_1 n - n^{1 - \varepsilon})} \sum_{i = 1}^n
		D_i \ind{Y_i > 0} + \frac{4}{nf(k)} \sum_{j = 1}^n D_j \ind{Y_j > 0} \\
		\label{eq:last_term_annr_ecm}
	&= \frac{2E_n}{(\nu_1 n - n^{1 - \varepsilon})} + \left(\frac{1}{(\nu_1 n - n^{1 - \varepsilon})}
		+ \frac{4}{n f(k)}\right)\sum_{j = 1}^n D_j \ind{Y_j > 0}.
\end{align}

Using computations similar to those leading up to \eqref{eq:bounding_theta_scaling_erased_edges} we get
\[
	\lim_{n \to \infty} \Prob{\frac{2E_n}{(\nu_1 n - n^{1 - \varepsilon})} > \frac{\delta}{4}, \Lambda_n} = 0.
\]
For the last term in (\ref{eq:last_term_annr_ecm}), we note that the terms in front of the summation are $O(n^{-1})$. Hence it suffices to prove that
for any $c > 0$
\begin{equation}\label{eq:annr_ecm_diffcult_term}
	\lim_{n \to \infty} \Prob{\frac{1}{n}\sum_{j = 1}^n D_j \ind{Y_j > 0} > c} = 0.
\end{equation}
For this we define $p = \frac{\gamma + 1}{2} < \gamma$ and $q = p/(p - 1)$. Then by first applying Markov's inequality
and then H\"{o}lder's inequality we get
\[
	\Prob{\frac{1}{n}\sum_{j = 1}^n D_j \ind{Y_j > 0} > c} \le \frac{1}{c} \Exp{\left|D_1 \ind{Y_1 > 0}\right|}
	\le \frac{\Exp{D_1^{p}}^{1/p}}{c} \Prob{Y_1 > 0}^{1/q}.
\]
To finish the argument we use \cite[Lemma 5.3]{chen2013}, which states that the probability that the number of erased
out-bound and in-bound stubs in the directed configuration model are positive, converges to zero. Their argument has a
straightforward extension to the undirected case so that $\lim_{n \to \infty} \Prob{Y_1 > 0} = 0$, which implies
\eqref{eq:annr_ecm_diffcult_term} and hence
\[
	\lim_{n \to \infty} \Prob{\eqref{eq:theta_bound_degree} > \frac{\delta}{2}, \Lambda_n} = 0,
\]
which finishes the proof.
\end{proof}

\bibliographystyle{plain}
\bibliography{references}

\end{document}